\documentclass[11 pt,pdftex]{article}
\pdfoutput=1
\usepackage[margin=2.5cm]{geometry}
\usepackage[all]{xy}

\usepackage{authblk, xcolor,graphicx,amsmath,amsthm,amssymb,amscd,enumerate,hyperref} 
\usepackage{tikz}
\usepackage{amsfonts}
\usepackage{pgf} 
\usepackage{mathrsfs}
\usepackage[utf8]{inputenc}

\usepackage{amsmath}
\usepackage{amsthm}
\usepackage{amsfonts}
\usepackage{amssymb}
\usepackage{graphicx}
\usepackage{multicol}
\usepackage[font=footnotesize,labelfont=bf]{caption}

\usetikzlibrary{arrows}
\newcommand*{\DashedArrow}[1][]{\mathbin{\tikz [baseline=-0.25ex,-latex, dashed,#1] \draw [#1] (0pt,0.5ex) -- (1.3em,0.5ex);}}%

\newlength{\matrixheight}

\newtheorem{thm}{Theorem}

\newtheorem{lem}[thm]{Lemma}
\newtheorem{cor}[thm]{Corollary}

\theoremstyle{definition}
\newtheorem{dfn}[thm]{Definition}

\newtheorem{rem}[thm]{Remark}

\newtheorem{que}[thm]{Question}

\DeclareMathOperator{\CH}{CH}

\DeclareMathOperator{\Hom}{Hom}
\DeclareMathOperator{\shHom}{\underline{Hom}}

\DeclareMathOperator{\hcf}{hcf}

\DeclareMathOperator{\Pic}{Pic}

\DeclareMathOperator{\Effbar}{\overline{Eff}}

\DeclareMathOperator{\gr}{gr}

\newcommand{\oo}{\mathcal{O}}
\newcommand{\QQ}{\mathbb{Q}}
\newcommand{\HH}{\mathbb{H}}

\newcommand{\RR}{\mathbb{R}}
\newcommand{\PP}{\mathbb{P}}
\newcommand{\ZZ}{\mathbb{Z}}
\newcommand{\NN}{\mathbb{N}}
\newcommand{\FF}{\mathbb{F}}
\newcommand{\CC}{\mathbb{C}}
\renewcommand{\AA}{\mathbb{A}}
\newcommand{\GG}{\mathbb{G}}

\newcommand{\TT}{\mathbb{T}}
\newcommand{\cE}{\mathcal{E}}
\newcommand{\cD}{\mathcal{D}}

\DeclareMathOperator{\Gm}{\CC^\times}

\begin{document}

\bibliographystyle{plain}

\title{Hyperelliptic Integrals and Mirrors of\\ the Johnson--Koll\'ar del Pezzo Surfaces}

\author{Alessio Corti\footnote{a.corti@imperial.ac.uk} }
\author{Giulia Gugiatti\footnote{g.gugiatti16@imperial.ac.uk}}
\affil{Department of Mathematics, Imperial College London, London SW7 2AZ}

\date{25\textsuperscript{th} January, 2019}

\maketitle

\begin{abstract} \noindent For all $k>0$ integer, we show explicitly that the hypergeometric function
\[  
\widehat{I}_k(\alpha)=\sum_{j=0}^\infty
\frac{\bigl((8k+4)j\bigr)!j!}{(2j)!\bigl((2k+1)j\bigr)!^2 \bigl((4k+1)j\bigr)!} \ \alpha^j
\]
is a period of a pencil of curves of genus $3k+1$. The function $\widehat{I}_k$ is the regularised $I$-function of the family of anticanonical del Pezzo hypersurfaces $X=X_{8k+4} \subset \PP(2,2k+1,2k+1,4k+1)$ and the pencil we construct is a candidate LG mirror of the elements of the family. The surfaces $X$ were first constructed by Johnson and Koll\'ar~\cite{MR1821068}. The main feature of these surfaces, which makes the mirror construction especially interesting, is that $|-K_X|=|\oo_X (1)|=\emptyset$; thus, there is no way to form a Calabi--Yau pair $(X,D)$ out of $X$. We also discuss the connection between our constructions and the work of Beukers, Cohen and Mellit \cite{MR3613122} on hypergeometric functions.
\end{abstract}

\tableofcontents{}
 
\section{Introduction}
\label{sec:introduction}

We begin by stating our results; then we briefly comment on the context
and on the methods of our proofs; we conclude with a few open questions. 

\subsection{Results}
\label{sec:results}
We state our results; more detail will be given in Section~\ref{sec:series} and in Section~\ref{sec:relation}.
\smallskip

For all $k>0$ integer, consider the hypergeometric function defined by
the power series:
\begin{equation}
\label{eq:I-function}  
\widehat{I}_k(\alpha)=\sum_{j=0}^\infty
\frac{\bigl((8k+4)j\bigr)!j!}{(2j)!\bigl((2k+1)j\bigr)!^2 \bigl((4k+1)j\bigr)!} \ \alpha^j
\end{equation} 
The function $\widehat{I}_k$ satisfies a hypergeometric
differential equation of order $6k+2$, \emph{see} Remark~\ref{rem:sing-rank}.

Conjecturally $\widehat{I}_k$ is (a shift of) Givental's
$\widehat{G}$-function --- a generating function of certain
Gromov--Witten invariants, \emph{see} Remark~\ref{rem:Gfunction} --- of the family of
anticanonical del Pezzo hypersurfaces 
\begin{equation} \label{eq:0a}
X_{8k+4} \subset \PP(2,2k+1,2k+1,4k+1) 
\end{equation}
These surfaces
form the main series of the classification of anticanonical
quasi-smooth and well-formed two-dimensional weighted hypersurfaces
\cite{MR1821068}, \emph{see} Section~\ref{sec:surfaces}.
\smallskip

In this paper we give positive answers to the following questions
which, as explained in Section~\ref{sec:mirrorsymmetry}, for us are
equivalent:
\begin{que}
Is $\widehat{I}_k$ a period of a pencil of curves?
\end{que}

\begin{que}
Does the family of anticanonical del Pezzo hypersurfaces of
Equation~\eqref{eq:0a} have a Landau--Ginzburg (LG) mirror?
\end{que}
\smallskip

For all $k>0$ integer, our answer to both questions is given by the
one-parameter family of hyperelliptic curves:
\begin{equation}
  \label{eq:1a}
 Y_k = \left(\alpha y^2 -h_{k,\alpha}(t_0, t_1) =0\right)\subset
 \PP(1,1,3k+2)\times \CC^\times  
\end{equation}
where $t_0, t_1, y$ are coordinates on $\PP(1,1,3k+2)$, $\alpha$ is
a coordinate on $\CC^\times$ and
\begin{equation}
\label{eq:1b} 
h_{k,\alpha}(t_0,t_1)=t_1 (4t_1^{2k+1}+\alpha t_0^{2k+1})
\bigl(-64t_1^{4k+2}+t_0t_1^{4k+1}-32\alpha
t_0^{2k+1}t_1^{2k+1}-4\alpha^2 t_0^{4k+2}\bigr) 
\end{equation}
together with the projection $w_k\colon Y_k\to \CC^\times$ to the second factor.

\begin{rem}
\label{rem:hyp}
Let 
\begin{equation}
  \label{eq:alpha0}
\alpha_{k,0}=\frac{{(4k+1)}^{4k+1}}{  4^{8k+3} {(2k+1)}^{2(2k+1)}  }
\end{equation}
For $\alpha\neq \alpha_{k,0}$, the fibre
  $Y_{k,\,\alpha}=w_k^{-1}(\alpha)$ is a nonsingular hyperelliptic curve of
  genus $3k+1$.  
\end{rem}

Our first result is:
\begin{thm}[Main Theorem]
\label{thm:main_theorem} For all $k>0$ integer, $\widehat{I}_k (\alpha)=\pi_k(\alpha)$, where
\begin{enumerate}[(A)]
\item $\widehat{I}_k(\alpha)$ is the hypergeometric function of Equation~\eqref{eq:I-function} above, and 
\item  $\pi_k(\alpha)$ is the period of the family $Y_k$ \eqref{eq:1a} given by:
\begin{equation} \label{eq:period}
\pi_k(\alpha) = \frac{1}{2\pi \text{\texttt{i}}} \oint_{\gamma_{k,\alpha}} \frac{ t^{2k}\,dt}{y }   
\end{equation}
where $\gamma_{k, \alpha}\subset Y_{k,\,\alpha}$  is the explicit cycle described in Section~\ref{sec:periods} and $t=t_1/t_0$. 
\end{enumerate}
\end{thm}

The solutions of the differential equation satisfied by
$\widehat{I}_k$ are sections of an irreducible complex local system on
$U_k=\PP^1\setminus \{0,\alpha_{k,0},\infty\}$, which we denote by
$\HH_k^\text{red}$, \emph{see} Remark~\ref{rem:sing-rank}. 

Write $Y_{U_k}=w_k^{-1}(U_k)$ and denote by $w_{U_k}\colon Y_{U_k}\to
U_k$ the restriction. 

By the main Theorem and Remark~\ref{rem:local_systems}
$\HH_k^\text{red}$ is a subquotient of the local system
$R^1w_{U_k \, \star} \CC$. Since the two local systems have the same rank, we
obtain:

\begin{cor}
  \label{cor:1} For all $k>0$ integer, 
   $\HH_k^\text{red}$ is isomorphic to $R^1w_{U_k \, \star} \CC$. 
\end{cor}

Our next result connects our construction to the work of Beukers,
Cohen and Mellit \cite{MR3613122}. More detail on this discussion can
be found in Section~\ref{sec:relation}.  For all $k>0$
integer, consider the manifold:
\begin{equation}
  \label{eq:2a}
  W_k = \left( \alpha \cdot (u_1+u_2+u_3+u_4-1) -
    {u_1}^2{u_2}^{2k+1}{u_3}^{2k+1}{u_4}^{4k+1}=0\right) \subset
  \TT^4\times \CC^\times 
\end{equation}
where $\TT^4\cong (\CC^\times)^4$ is a $4$-dimensional torus with coordinates
$u_1,\dots, u_4$ and $\alpha$ is a coordinate on $\CC^\times$. Denote
by $\upsilon_k\colon W_k \to \CC^\times$ the second projection. The
3-dimensional hypersurface
$W_{k, \alpha} =\upsilon_k^{-1}(\alpha)\subset \TT^4$ is nonsingular if
and only if $\alpha\neq \alpha_{k,0}$.

A very special case of the work~\cite{MR3613122} relates point
counting in characteristic $p$ on $W_{k,\alpha}$ to finite analogs of
the function $\widehat{I}_k$ of Equation \eqref{eq:I-function}. This
result strongly suggests that the hypergeometric function
$\widehat{I}_k$ is a period of the family of \mbox{$3$-folds}
$W_{k,\alpha}$ ($\alpha \in U_k$) or, equivalently, that the variation
$H^3(W_{k,\alpha},\QQ)$ is related to the local system
$\HH_{k}^\text{red}$.\footnote{It would not be very difficult --- but
  it would take us too far --- to prove that $\widehat{I}_k$ is a
  period of the family of Equation~\eqref{eq:2a}.} In
Section~\ref{sec:relation} we prove:

\begin{thm}
  \label{thm:2a} For all $k>0$ integer, write 
  $W_{U_k}=\upsilon_k^{-1}(U_k)$ and
  denote by $\upsilon_{U_k}\colon W_{U_k}\to U_k$ the restriction. Then:  
\begin{equation}  \label{eq:idea} 
\mathrm{gr}_3^W R^3 \upsilon_{U_k\, !} \QQ\, (1)=R^1w_{U_k \, \star} \QQ
\end{equation} 
\end{thm}

This gives an interpretation of a special case of the results in
\cite{MR3613122} in terms of mirror symmetry.

\subsection{Context and a few words on our proofs}
\label{sec:cont-furth-direct}
Our mirrors of the surfaces
$X=X_{8k+4}\subset \PP(2, 2k+1, 2k+1, 4k+1)$ are not covered by any
mirror construction known to us. Indeed, since
$H^0(X, -K_X)=H^0(X,\mathcal{O}_X(1))=(0)$, there is no Calabi--Yau
pair $(X,D)$, and hence the intrinsic mirror symmetry program
\cite{2016arXiv160900624G} is not applicable in this context.
\smallskip

By \cite[Theorem~5.4.4]{MR1081536} the local systems
$\HH^\text{red}_k$ support a canonical rational variation of Hodge
structures (VHS).\footnote{In fact, \cite{MR1081536} constructs an
  explicit geometric realisation of this VHS, different from the one
  of Equation~\eqref{eq:2a}.} 
  By the criterion of \cite{MR2824960,
  2015arXiv150501704F}, the variation has Hodge weight one. Thus, it
is natural to ask --- even as there is no reason to expect it --- if
$\HH^\text{red}_k$ is a (direct summand of) the variation of $H^1$ of
a one-parameter family of curves: this motivates our Question~1. Since
an irreducible local system supports at most one rational VHS,
\emph{see}~\cite[Proposition~2.1]{MR3484368}, our main Theorem implies
that $\HH^\text{red}_k=R^1w_{U_k\,\star} \QQ$ as VHS.
\smallskip

The general shape of the Fano/LG correspondence suggests that the
mirror of the family of anticanonical del Pezzo hypersurfaces
$X_{8k+4}$ is a function $w_k \colon Y_k\to \CC^\times$ with
one-dimensional fibres, smooth over $U_k$, together with an
identification of a subquotient of $R^1w_{U_k\, !} \CC$
with the hypergeometric local system of solutions of the differential equation
satisfied by $\widehat{I}_k$.  This motivates our Question~2.
\smallskip

Our proof of the main Theorem is elementary: we expand the period in
power series with the help of the residue theorem. The key difficulty
is to find the equation of $Y_k$ (and the integration cycles
$\gamma_k$).
\smallskip

The work \cite{MR3613122} and in particular the pencil of
\mbox{$3$-folds} $\upsilon_k\colon W_k \to \CC^\times$ of
Equation~\eqref{eq:2a} are the starting point for our
investigations. Morally, this work identifies the local system
$\HH^\text{red}_k$ with $\mathrm{gr}_3^W R^3\upsilon_{U_k\, !} \CC$
--- more detail on this point can be found in
Section~\ref{sec:relation} --- and thus provides a mirror \emph{of the
  wrong dimension}.  Then it is natural to ask if there is a morphism
$w_k\colon Y_k \to \CC^\times$ with one-dimensional fibres such that
$\mathrm{gr}_3^W R^3 \upsilon_{U_k\, !} \QQ\, (1)=R^1w_{U_k \, !}
\QQ$.
Our Theorem~\ref{thm:2a} states that our family of hyperelliptic
curves indeed has this property. The constructions in the proof of the
Theorem in Section~\ref{sec:proof} make it clear how the curve
$Y=Y_{k,\alpha}$ arises naturally from a study of the geometry of the
\mbox{$3$-fold} $W=W_{k, \alpha}$ $\forall \alpha \neq \alpha_{k, 0}$. The
fact that the variation $\mathrm{gr}_3^W H^3(W_{k,\alpha},\QQ)$ is the
variation of the $H^1$ of a pencil of curves is nontrivial. There is
no reason to expect it. Looking for an explanation, one is lead to
wonder whether the \mbox{$3$-folds} $W_{k, \alpha}$ are rational. In
Section~\ref{sec:conic-bundle-struct} we construct a conic bundle in
the Mori category, birational to $W_{k, \alpha}$. Interestingly, a
study of the geometry of this conic bundle suggests that
$W_{k,\alpha}$ is not rational, and that the conic bundle may even be
birationally rigid.

\smallskip

The idea of the proof of Theorem~\ref{thm:2a} is to construct a
partial compactification $ W \subset \widehat{W}$ with a del Pezzo
fibration $\phi \colon \widehat{W} \to \CC^\times$. The del Pezzo
fibration becomes visible after a monomial substitition of
coordinates, \emph{see} Equation~\eqref{eq:change}. In the Appendix, we
explain how we discovered the del Pezzo fibration by running a minimal
model program for a partial resolution of a compactification of
$W$. Once we have the del Pezzo fibration, the proof of
Theorem~\ref{thm:2a} is an exercise in mixed Hodge theory for which
models exist in the literature. The key point is
Lemma~\ref{lem:final}, which constructs an algebraic cycle
$Z\subset \CH_2 ( Y \times_{\CC^\times} \widehat{W} )$ inducing an
isomorphism $p_\star \QQ_{Y} \to R^2\phi_\star \QQ_{\widehat{W}}(1)$.
\smallskip

It is natural to ask if our computation of the cohomology of $W$ in
terms of the del Pezzo fibration
$\phi \colon \widehat{W}\to \CC^\times$ really is the easiest way to
prove Theorem~\ref{thm:2a}. We think that it is. The conic bundle
birational to $W$ that we construct in
Section~\ref{sec:conic-bundle-struct} gives another, more
complicated, way to understand the geometry of $W$.

\subsection{Further questions}
\label{sec:further-question}

It may be possible to determine all hypergeometric functions that are
periods of a pencil of curves, and describe the pencils explicitly. 
\smallskip 

There are several contexts where a construction of mirror symmetry is
available that produces a mirror of the wrong dimension: one of these
is general toric complete intersections \cite{Hori-Vafa, MR3613122};
another is the abelian/nonabelian correspondence \cite{MR2367022}. Our
ideas here can form the basis of a systematic method to extract from
these constructions mirrors of the correct dimension. In particular it
would be extremely attractive to eliminate the assumptions of
Hori--Vafa \cite{Hori-Vafa} and obtain mirrors for all toric complete
intersections.

\smallskip

It will be interesting to see that our mirrors satisfy Homological
Mirror Symmetry.

\subsection{Structure of the paper}
\label{sec:plan-paper}

All sections in this paper logically depend on the Introduction; other
than that, they are logically mutually independent and can be read (or
not read, as the case may be) in any order. 

The paper is structured as follows. In Section~\ref{sec:series} we
expand on the notion of mirror symmetry for Fano anticanonical
weighted hypersurfaces. The families of surfaces \eqref{eq:0a} are an example
of such hypersurfaces; our main Theorem can be interpreted as giving
mirrors of these families. In Section~\ref{sec:periods} we prove the
main Theorem.  In Section~\ref{sec:relation} we show how our
Equation~\eqref{eq:2a} arises from the work \cite{MR3613122} as a mirror
of the wrong dimension. In Section~\ref{sec:proof} we prove
Theorem~\ref{thm:2a}. In the Appendix we explain how we discovered the
del Pezzo fibration $\phi \colon \widehat{W}\to \CC^\times$. The
Appendix ends with the construction of a compactification
$W \subset W^\prime$ with a conic bundle structure
$\psi \colon W^\prime \to \FF_1$.

\subsection{Acknowledgments}
\label{sec:acknow}

It is a pleasure to thank Fernando Rodriguez Villegas for attracting our
attention to the work of Beukers, Cohen and Mellit \cite{MR3613122},
and Tom Coates and Don Zagier for helpful conversations.  A.C. is partially
supported by EPSRC Program Grant \emph{Classification, Computation and
  Construction: New Methods in Geometry}; G.G. is supported by
EPSRC-funded Centre for Doctoral Traning \emph{The London School of
  Geometry and Number Theory}.


\section{Anticanonical weighted hypersurfaces and mirror symmetry}
\label{sec:series}

In this Section we expand on the notion of mirror symmetry for Fano
anticanonical weighted hypersurfaces. The families of surfaces
$X=X_{8k+4}$ of the Introduction are an example of such hypersurfaces;
our main Theorem can be interpreted as giving mirrors of these
families.

\medskip

We denote by  $\PP(a_0, a_1, \dots , a_m)$ the weighted projective $m$-space with weights $a_0\leq a_1 \leq \cdots \leq a_m$; we simply write $\PP$ when $m$ and the weights are clear from the context.

For $X=X_d \in |\mathcal{O}_{\PP}(d)|$ a quasismooth~\cite[6.3]{MR1798982} and wellformed~\cite[6.10]{MR1798982} hypersurface of degree $d$, the adjuction formula for the canonical sheaf $K_X$ \cite[6.14]{MR1798982} states that
\[K_X = \mathcal{O}_X(d-\sum_{i=0}^ma_i) \,.
\] 
By definition $X$ is Fano if and only if $-K_X$ is ample, i.e., if and only if $d < \sum_{i=0}^ma_i$. We say that $X$ is anticanonical if $-K_X=\oo_X(1)$, that is, $d=\sum_{i=0}^ma_i-1$.

\subsection{Our notion of mirror symmetry}
\label{sec:mirrorsymmetry}
We state what we mean by mirror for a quasismooth wellformed Fano anticanonical weighted hypersurface 
$X=X_d \subset \PP(a_0, \dots a_m)$ of dimension $n=m-1$.

\begin{dfn}
  \label{dfn:1}
The regularised $I$-function of $X$ is defined as the hypergeometric series
  \begin{equation}
    \label{eq:3}
  \widehat{ I}_X(\alpha)= \sum_{j=0}^{\infty} \frac{ (d \, j)! \, j! }{(a_0j)! \dots (a_m j)!} \, \alpha^{ j} \quad \big(\alpha \in \CC \big)
\end{equation}
\end{dfn}

\begin{rem}
\label{rem:Gfunction}
The paper \cite[Sections B,C]{2013arXiv1303.3288C} defines the
$G$-function of $X$, a generating series for certain Gromov--Witten
invariants of $X$, and the $I$-function of $X$.\footnote{In
  \cite{2013arXiv1303.3288C} $X$ is a smooth 
  variety. Here we think of a quasi-smooth well-formed weighted
  hypersurface as a smooth Deligne--Mumford stack. The definitions of
  $I_X$ and ${G}_X$ make sense in this more general context.}  The
function $\widehat{I}_X$ \eqref{eq:3} is the Fourier transform of the
$I$-function of $X$.  Conjecturally,
${G}_X(\alpha)=e^{-c\alpha} {I}_X(\alpha)$, where $c$ is the only
rational number such that the right-hand side has the form
$1+O({\alpha}^2)$, \emph{see} \cite[Proposition
D.9]{2013arXiv1303.3288C}.
\end{rem}

\begin{rem}
\label{rem:Hypergeometric_Local_system}
Our functions $\widehat{I}_X$ satisfy a hypergeometric differential
equation on $\PP^1$, nonsingular outside
$\Sigma=\{0,\alpha_0,\infty\}$ where
$\alpha_0=\frac{\prod a_i^{a_i}}{d^d}$, whose solutions are the
sections of an irreducible complex local system on
$\PP^1\setminus \Sigma$ which we denote by $\HH_X^\text{red}$.  To be
a little more specific, write
\[ P_0(j)= - \prod_{i=0}^{m} (a_ij) (a_ij-1) \cdots (a_ij -a_i+1)   \quad \text{and} \quad P_1(j-1)=  j \ (d \, j)(d \, j -1)  \cdots (d \, j -d +1)
\]
Consider the differential operator $H_X= P_0(D) + \alpha P_1(D) \in \ZZ[\alpha, D]$ where $D=\alpha \frac{d}{d\alpha}$, and denote by $H^\text{red}_X$ the operator obtained removing from both $P_0(j)$ and $P_1(j-1)$ a copy of every common factor. By \cite[Corollary~3.2.1]{MR1081536} $H^\text{red}_X$ is irreducible and it is easy to see that $H^\text{red}_X\cdot \widehat{I}_X=0$.
\end{rem}

\begin{dfn} \label{dfn:LGmodel} A \emph{Landau--Ginzburg (LG) model} is a tuple $(Y^n, w, \omega, \gamma)$ where:
\begin{enumerate}[(i)]
\item $Y$ is a smooth algebraic manifold of dimension $n$;
\item $w\colon Y \to \CC^\times$ is a quasi-projective morphism. For $\alpha \in \CC^\times$ we denote by $Y_\alpha=w^{-1}(\alpha)$ the fibre. 
\end{enumerate}
Denote by $U\subset \CC^\times$ the set of regular values of $w$, and by $w_U\colon Y_U=w^{-1}(U) \to U$ the restriction.
\begin{enumerate}[(i)]
\item[(iii)] $\omega\in \Gamma (U, w_{U\, \star}\Omega^{n-1}_{Y_U/U})$. For $\alpha \in U$, we write $\omega_\alpha \in H^0(Y_\alpha, \Omega^{n-1}_{Y_\alpha})$ the corresponding form;
\item[(iv)] $\gamma \in \Gamma (D^\times,R^{n-1}w_{U\, !} \QQ_{Y_U})$ where $0\in D \subset \CC$ is a small disk and $D^\times =D\setminus \{0\}$. For $\alpha \in D^\times$ we denote by $\gamma_\alpha \in H_{n-1}(Y_\alpha, \QQ)$ the corresponding cycle.
\end{enumerate}
 The \emph{period} of the LG model is the function
\begin{equation} \label{eq:periodLG}
\pi(\alpha)=\int_{\gamma_\alpha} \omega_\alpha
 \quad (\alpha \in D^\times)
\end{equation}
\end{dfn}

\begin{dfn} \label{mirror} 
Let $X=X_{d} \subset \PP(a_0, \dots a_m)$ be a quasismooth wellformed Fano anticanonical weighted hypersurface of dimension $n=m-1$.
A LG model $(Y^n, w, \omega, \gamma)$ is \emph{mirror} of $X$ if for
all $\alpha$
\[
\widehat{I}(\alpha) = \pi(\alpha)
\]
\end{dfn}

\begin{rem}
  \label{rem:local_systems}
It follows directly from  Definition \ref{mirror} that $\HH_X^\text{red}$ is a subquotient of $R^{n-1} w_{U\, !} \CC$ and this fact could be taken as a weak version of mirror symmetry.

Indeed, write
\[
\cE=(R^{n-1}w_{U\, \star} \QQ)\otimes \oo_U.
\]
The sheaf $\cE$ carries an algebraic connection $\nabla \colon \cE \to \Omega^1_U \otimes \cE$ (the Gauss--Manin connection) by means of which we can regard it as a $\cD_U$-module, where $\cD_U$ denotes the sheaf of differential operators on $U,$ and
\[
R^{n-1}w_{U\, \star} \CC=\shHom_{\cD_U}(\oo_U, \cE)\,,
\quad\text{and}\quad
R^{n-1}w_{U\, !} \CC=\shHom_{\cD_U}(\cE, \oo_U)
\]
are the local systems of flat sections and of solutions of $\cE$. Now $w_{U\,\ast}\Omega^{n-1}_{Y_U/U}\subset \cE$ as the last piece of the Hodge filtration and we regard $\omega$ as a section of $\cE$ by means of this inclusion. On the other hand we regard $\gamma$ as a solution of $\cE$ and recover the period as $\pi=\gamma (\omega)$. We have an inclusion and a surjection of $\cD_U$-modules:
\[
\cD_U \cdot \omega \subset \cE\,,
\quad \text{and}\quad
\gamma\colon \cD_U\cdot \omega \to \cD_U\cdot \pi 
\]
So we have  $\HH_X^\text{red}=\shHom_{\cD_U}(\cD\cdot \pi, \oo_U)\subset \shHom_{\cD_U}(\cD_U\cdot \omega, \oo_U)$ and  
$R^{n-1}w_{U\, !} \CC=\shHom_{\cD_U}(\cE, \oo_U)\twoheadrightarrow \shHom _{\cD_U}(\cD_U\cdot \omega, \oo_U)$.

\end{rem}

\subsection{Anticanonical del Pezzo hypersurfaces}
\label{sec:surfaces}
In this paper we call a Fano surface a del Pezzo surface.
\smallskip

Johnson and Koll\'ar \cite{MR1821068} classify all anticanonical quasismooth wellformed del Pezzo surfaces in weighted projective $3$-spaces. Their classification consists of $22$ sporadic cases and the series \eqref{eq:0a}, where  $k \in \NN, \  k > 0$. The 22 sporadic cases are all listed in~\cite[Theorem 8]{MR1821068}.
\smallskip

By \eqref{eq:3} for all $k >0$ integer the regularised $I$-function of any surface of the family \eqref{eq:0a} is given by Equation \eqref{eq:I-function}.

\begin{rem} \label{rem:sing-rank}
  By Remark \ref{rem:Hypergeometric_Local_system} the function $\widehat{I}_k$ \eqref{eq:I-function}  satisfies an hypergeometric differential equation
  on $\PP^1$ which is singular on $\Sigma_k=\{0, \alpha_{k,0}, \infty\}$, where $\alpha_{k,0}$ is as in Equation~\eqref{eq:alpha0}. The reduced differential operator associated to $\widehat{I}_k$, which we denote by $H_k^{\mathrm{red}},$ has order $6k+2,$ thus the local system $\HH_k^{\mathrm{red}}$ given by its solutions has rank $6k+2.$
\end{rem}

\begin{rem}
In light of the definitions of Section \ref{sec:mirrorsymmetry}, the data at the beginning of Section~\ref{sec:periods} define a LG model and our main Theorem can be interpreted as stating that this LG model is the mirror to the corresponding family of Johnson and Koll\'ar. Also, by Remarks \ref{rem:local_systems} and \ref{rem:sing-rank}, $\HH_k^{\mathrm{red}} \simeq R^1w_{U_k \, \star} \CC$, as stated in Corollary \ref{cor:1}.
\end{rem}


\section{Proof of the main Theorem}
\label{sec:periods}

In this Section, we prove the main Theorem stated in the
Introduction. We begin by giving data to construct the period
integral. 

\paragraph{The period integral} Fix an integer $k >0$. We define data
$(Y_k, w_k, \omega_k, \gamma_k)$ as follows:
\begin{enumerate}[(i)]
\item $Y_k \subset \PP(1,1,3k+2) \times \CC^{\times}$ is the $2$-dimensional manifold given by \eqref{eq:1a} and \eqref{eq:1b};
\item $w_k \colon Y_k \to \CC^{\times}$ is the projection on the second factor.
\end{enumerate}
Let $\alpha_{k,0}$ be as in Equation \eqref{eq:alpha0} and consider $\alpha \neq \alpha_{k,0}$.
\begin{enumerate}[(i)]
\item[(iii)]  $\omega_{k, \alpha}=\frac{1}{2\pi \text{\texttt{i}}} \frac{ t^{2k}\,dt}{y}$, where $t=t_1/t_0$;
\item[(iv)] $\gamma_{k,\alpha}$ is the cycle that we describe next.
\end{enumerate}
From this data we construct the period integral:
\[
\pi_k(\alpha) = \frac{1}{2\pi \text{\texttt{i}}}
\oint_{\gamma_{k,\alpha}} \omega_{k,\alpha} 
\]
This is the period of $Y_k$ of Equation~\eqref{eq:period}.

\begin{rem}
This data defines a LG model, according to Definition~\ref{dfn:LGmodel}, and the period~\eqref{eq:period} is the period of the LG model~\eqref{eq:periodLG}. Our main
Theorem can be interpreted as stating that this LG model is a mirror of the family of surfaces $X_k$. 
\end{rem}

\paragraph{The cycle of integration} Let us denote by
$p_{k,\alpha} \colon Y_{k, \alpha}\to \PP^1$ the $2:1$ cover;
$p_{k,\alpha}$ is branched at the $6k+4$ roots of the polynomial \
\[
h_{k,\alpha}(t)=t \left( 4t^{2k+1}+\alpha \right)
\left(-64t^{4k+2}+t^{4k+1}-32\alpha t^{2k+1}-4\alpha^2 \right)
\]
\begin{lem}
  \label{lem:asymptotics}
 For $|\alpha |\ll 1$ the polynomial $h_{k,\alpha}$ has:
 \begin{itemize}
 \item A root at $t=0$;
 \item $4k+1$ roots of norm $\sim |\alpha|^{\frac{2}{4k+1}}$;
 \item $2k+1$ roots of norm $\sim |\alpha|^{\frac{1}{2k+1}}$;
 \item A root of norm $\sim \frac{1}{64}$.
 \end{itemize}
\end{lem}

\begin{proof}
  Clearly $h_{k, \alpha}$ has a root at $t=0$ and $2k+1$ roots of norm
  $|\alpha/4|^{\frac{2}{4k+1}}$. Consider the polynomial
  $g_{k, \alpha}(t)=-64t^{4k+2}+t^{4k+1}-32\alpha t^{2k+1}-4\alpha^2$;
  then as $\alpha \to 0$ $g_{k,\alpha}( t)$ has a root
  $t_{\alpha} \sim \frac{1}{64}$ and $4k+1$ roots $t_\alpha$ of norm
  $\sim |\alpha|^{\frac{2}{4k+1}}$. Indeed
\[
\lim_{\alpha\to 0} g_{k,\alpha}(t)=t^{4k+1}(-64t +1)
\]
hence for $|\alpha |\ll 1$ $g_{k,\alpha}$ has a root that tends to $ \frac{1}{64}$ and $4k+1$ roots that tend to $0$. 
Now $t_\alpha$ is a root if and only if
\[
t_\alpha^{4k+1}=4(4t_\alpha^{2k+1} +\alpha)^2
\]
and if $t_\alpha \to 0$, then $|t_\alpha|^{4k+1} \sim |\alpha|^2$.
\end{proof}

Choose a continuous function $\rho_k\colon (0,1)\to \RR$ such that for $r\ll 1$:
\[
r^{\frac{2}{4k+1}}\ll \rho_k (r) \ll r^{\frac{1}{2k+1}}
\]
For $|\alpha|\ll 1$ let $\overline{\gamma}_{k,\alpha}$ be the circle of radius $\rho_k(|\alpha|)$
around the origin in $\CC$ starting at $t_0=\rho_k (|\alpha|)$.
By Lemma \ref{lem:asymptotics} this circle divides $\PP^1$ into two
regions each containing an even number of branch points of
$p_{k,\alpha}$, \emph{see} Figure \ref{fig:circle}. Hence
$p_{k, \alpha}^{-1}(\overline{\gamma}_{k,\alpha})\subset Y_{k,\alpha}$ consists of two
disjoint circles: we take $\gamma_{k,\alpha}$ to be the lift
along which the power series expansion in Equation
\eqref{eq:expansion} is valid, \emph{see} Figure
\ref{fig:cover}. \footnote{Alternatively we choose the base point for
  our lift such that $y>0$.}

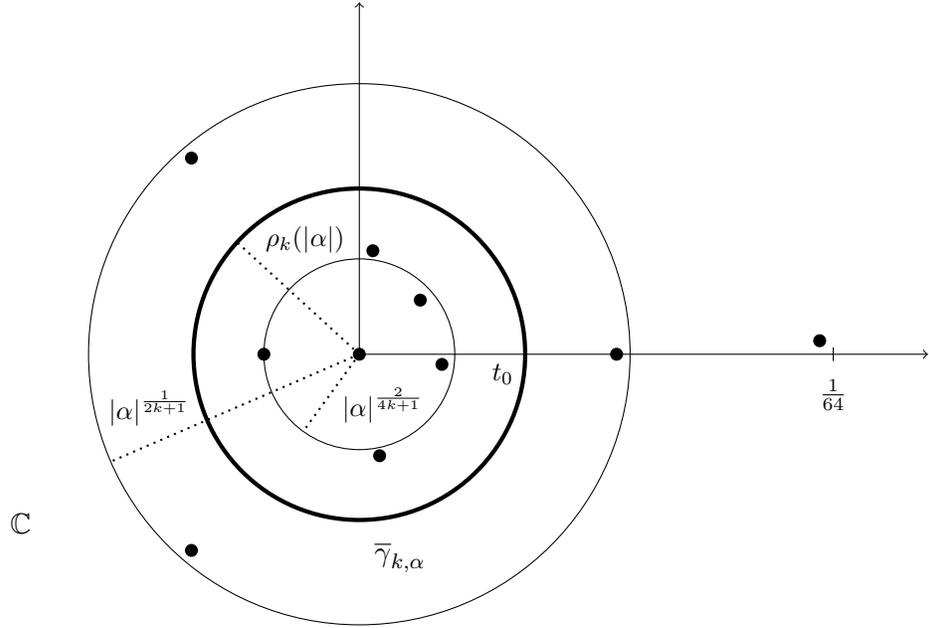
\begin{figure}[!ht] \centering
  \begin{tikzpicture}[scale=0.9] 
    \clip(-5,-0.5) rectangle (14,10);
    
 \draw[ultra thick] (4,4) circle [x radius=2.45cm, y radius=2.45cm]; 
 \draw (4,4) circle [x radius=4cm, y radius=4cm]; 
 \draw (4,4) circle [x radius=1.41cm, y radius=1.41cm];  

  \draw [<->] (4,9.2) -- (4,4) -- (12.4,4); 
 
  \draw [dotted, thick] (4,4) -- (2.2,5.65); 
  \draw [dotted, thick] (4,4) -- (3.2, 2.9);
  \draw [dotted, thick] (4,4) -- (0.3,2.4);
  
 \draw[fill, thick] (4,4) circle [radius=0.08]; 
  \draw[fill, thick] (10.8,4.2) circle [radius=0.08]; 
  \draw[fill, thick] (7.8,4) circle [radius=0.08];  
  \draw[fill, thick] (1.52,6.9) circle [radius=0.08];
  \draw[fill, thick] (1.52,1.1) circle [radius=0.08];
  \draw[fill, thick] (5.22,3.85) circle [radius=0.08]; 
  \draw[fill, thick] (2.59,4) circle [radius=0.08];
   \draw[fill, thick] (4.2,5.53) circle [radius=0.08];
 \draw[fill, thick] (4.3,2.5) circle [radius=0.08];
 \draw[fill, thick] (4.9,4.8) circle [radius=0.08];
  
 \node at (6.12,3.7) {\small{$t_0$}}; 
 \draw[] (11,4.1) -- (11, 3.9); 
 \node at (11,3.4) {\small{$\frac{1}{64}$}};
 \node at (-1,1.5) {$\mathbb{C}$};
 \node at (4.6, 1) {$\overline{\gamma}_{k,\alpha}$};
 \node at (3.2,5.7) {\small{$\rho_k(|\alpha|)$}};
 \node at (4.35,3.25) {\small{${|\alpha|}^\frac{2}{4k+1}$}};
 \node at (0.9,3.2) {\small{${|\alpha|}^\frac{1}{2k+1}$}};
  \end{tikzpicture}
  \caption{The circle $\overline{\gamma}_{k,\alpha} \subset \CC.$ The
    $10$ marked points represent the roots of $h_{k,\alpha}$ for
    $k=1.$} \label{fig:circle} 
  \end{figure}

 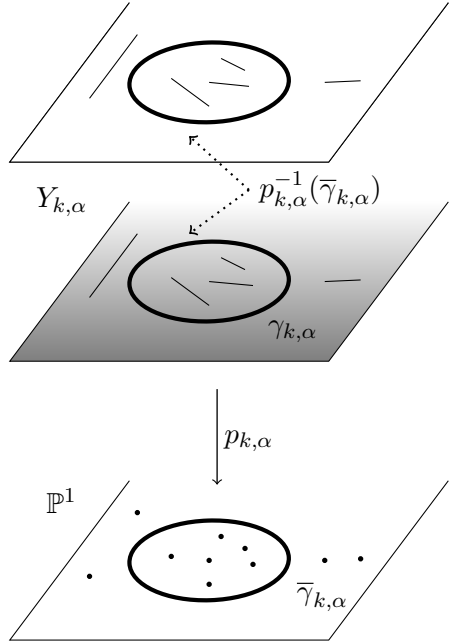
\begin{figure}[!ht] \centering
   \begin{tikzpicture} [scale=0.53] 
     \clip(-1,-12.8) rectangle (13,5.2);
  \draw (3,4) -- (0,0) -- (8,0) -- (11,4);  
  \draw[shade, top color=white, bottom color=white] (3,4) -- (0,0) -- (8,0) -- (11,4); 
  \draw[shade, top color=white, bottom color=gray]  (3,-1) -- (0,-5) -- (8,-5) -- (11,-1);  
  \draw[shade, top color=white, bottom color=white]  (3,-8) -- (0,-12) -- (8,-12) -- (11,-8);
  \draw[->][line width=0.5] (5.2, -5.7) -- (5.2, -8.1); 
  \node at (6,-7) {{$p_{k,\alpha}$}};
  \node at (1.3,-8.5) {{$\mathbb{P}^1$}};
  \node at (1.3,-1) {{$Y_{k,\alpha}$}};
  
 \draw[ultra thick] (5,2) circle [x radius=2cm, y radius=1cm, rotate=2]; 
 \draw[ultra thick] (5,-3) circle [x radius=2cm, y radius=1cm, rotate=2];
 \draw[ultra thick] (5,-10) circle [x radius=2cm, y radius=1cm, rotate=2];
 \draw[thin]  (5,2) -- (6,1.9);  
 \draw[thin] (5.3, 2.6) -- (5.9, 2.3);
 \draw[thin] (4.05,2.1) -- (5,1.4);  
 \draw[thin] (7.9,2) -- (8.8,2.04); 
 \draw[thin] (3.2,3.2) --  (2,1.6);
 \draw[thin]  (5,-3) -- (6.1,-3.1);  
 \draw[thin] (5.3, -2.4) -- (5.9, -2.7);
 \draw[thin] (4.05,-2.9) -- (5,-3.6);  
 \draw[thin] (7.9,-3) -- (8.8,-2.96); 
 \draw[thin] (3.2,-1.8) --  (2,-3.4);
 \draw[fill, thick] (5,-10) circle [radius=0.04]; 
 \draw[fill, thick] (6.1,-10.1) circle [radius=0.04];
 \draw[fill, thick] (5.3,-9.4) circle [radius=0.04];
 \draw[fill, thick] (5.9,-9.7) circle [radius=0.04];
 \draw[fill, thick] (4.05,-9.9) circle [radius=0.04];
 \draw[fill, thick] (5,-10.6) circle [radius=0.04];
 \draw[fill, thick] (7.9,-10) circle [radius=0.04]; 
 \draw[fill, thick] (8.8,-9.96) circle [radius=0.04];
 \draw[fill, thick] (3.2,-8.8) circle [radius=0.04]; 
  \draw[fill, thick] (2,-10.4) circle [radius=0.04];
 \draw [->] [dotted, thick] (6,-0.7) -- (4.5,0.65);
 \draw [->] [dotted, thick] (6,-0.7) -- (4.5,-1.8);
 \node at (7.8,-0.72) {$p_{k,\alpha}^{-1}(\overline{\gamma}_{k,\alpha})$};
 \node at (7.84,-10.9) {$\overline{\gamma}_{k,\alpha}$};
\node at (7.1,-4.25) {$\gamma_{k,\alpha}$};  
\end{tikzpicture}
\caption{The circle $\gamma_{k,\alpha}$ on $Y_{k,\alpha}$. The gray sheet of the cover represents the one where \eqref{eq:expansion} is valid. The $5$ line segments on each sheet indicate a choice of branch cuts for $k=1$.}  \label{fig:cover}
\end{figure}

\paragraph{Proof of the main Theorem} An immediate consequence of the
next two lemmas. \qed

\begin{lem}
Let $\pi_k(\alpha)$ be the period integral described above. Then, setting $m=(2k+1)j$:
 \begin{equation}
  \label{eq:expansion}
\pi_k(\alpha)=
 \sum_{j=0}^{\infty}   \binom{-1/2}{m}(-1)^{m} 4^{3m-j}\sum_{p=j}^{2m} \binom{-1/2}{p-j} \binom{2m}{p}   \ \alpha^j
 \end{equation}
\end{lem}

\begin{proof}
By a small manipulation we write the period as:
 \begin{equation} \label{eq:period_simple}
  {\pi}_k(\alpha)= \frac{1}{2\pi \text{\texttt{i}}} \oint_ {\overline{\gamma}_{k,\alpha}}   \frac{1}{\sqrt{  \left(1 + \frac{4t^{2k+1}}{\alpha} \right) }} \frac{1}{\sqrt{\left(1-t \ \left(8+\frac{2\alpha}{t^{2k+1}}\right)^2 \right) }} \ \frac{  dt}{t}.
   \end{equation}
By the defining inequalities of the function $\rho_k$, and our choice of $\gamma_{k,\alpha}$, both the following power series expansions hold:
\begin{equation*}
\begin{split} 
\left( 1 +\frac{4t^{2k+1}}{\alpha}\right)^{-\frac1{2}}& = \sum_{n=0}^{\infty} \binom{-1/2}{n} \frac{4^nt^{(2k+1)n}}{\alpha^n}\\
\left(1-t \ \left(8+\frac{2\alpha}{t^{2k+1}}\right)^2\right)^{-\frac1{2}} &  =\sum_{m=0}^{\infty} \binom{-1/2}{m} (-1)^m 4^{3m} t^m {\left(1+\frac{\alpha}{4t^{2k+1}}\right)}^{2m}
\end{split}
\end{equation*} 

Plugging the two power series in~\eqref{eq:period_simple}, switching the series and the integral signs and using the binomial theorem we obtain:
\begin{equation*}
\begin{split}
  \pi_k(\alpha) =\sum_{n=0}^{\infty} \ \sum_{m=0}^{\infty} \
  \sum_{p=0}^{2m} \binom{-1/2}{n} \binom{-1/2}{m} \binom{2m}{p} (-1)^m
  4^{3m+n-p} \alpha^{p-n} \frac{1}{2\pi \text{\texttt{i}}} \oint_
  {\overline{\gamma}_{k,\alpha}} \frac{dt}{t} t^{m-(2k+1)(p-n)}.
\end{split}
\end{equation*}
Finally, applying the residue theorem around $t=0$ and setting $j=p-n,$ we obtain the result.
\end{proof}

\begin{lem}
For all  $k>0$ and $j\geq 0$ integers, setting $m=(2k+1)j$:
\begin{equation} \label{eq:coefficients}
\binom{-1/2}{m}(-1)^{m} 4^{3m-j}\sum_{p=j}^{2m} \binom{-1/2}{p-j} \binom{2m}{p}= \frac{\bigl((8k+4)j\bigr)!j!}{(2j)!\bigl((2k+1)j\bigr)!^2 \bigl((4k+1)j\bigr)!} 
\ \, 
\end{equation}
\end{lem}

\begin{proof} Clearly \eqref{eq:coefficients} holds when $j=0,$ thus we assume $j \geq 1$. In what follows we repeatedly use the identity: 
  \begin{equation}  \label{eq:bifactorial} 
(2l-1)!! = \frac{(2l)!}{2^l \cdot l!} \quad \left(l>0 \ \text{integer}\right)
\end{equation}
We have: 
\begin{equation} \label{eq:binomial} \binom{-1/2}{m}=\frac{(-1)^m (2m-1)!!}{2^mm!} = \frac{(-1)^m}{4^m} \ \binom{2m}{m} \end{equation}
We set $i=p-j$ and we write the finite sum on the left hand side of \eqref{eq:coefficients} as:
\begin{equation} \label{eq:finitesum}
\sum_{p=j}^{2m} \binom{-1/2}{p-j} \binom{2m}{p}=\sum_{i=0}^{2m-j} \binom{-1/2}{i}\binom{2m}{j+i}=\sum_{i=0}^{2m-j} \binom{-1/2}{i} \binom{2m}{2m-j-i}=\binom{2m-1/2}{2m-j} \end{equation}
where the last equality in \eqref{eq:finitesum} follows from the Chu--Vandermonde formula for generalised binomial coefficients:
\begin{equation*} \label{eq:CV}
  \sum_{i=0}^{n} \binom{\beta}{i} \binom{\alpha}{n-i}=\binom{\beta+\alpha}{n} \quad \left( \alpha, \beta \in \mathbb{C}, \  n \in \NN  \right)
\end{equation*} 
Plugging \eqref{eq:binomial} and \eqref{eq:finitesum} in \eqref{eq:coefficients} and using that $m=(2k+1)j$ we can rewrite \eqref{eq:coefficients} as:
\begin{equation}  \label{eq:coefficients1}  4^{(4k+1)j} \  \left((4k+2)j\right)!   \ \binom{(4k+2)j-1/2}{(4k+1)j}=\frac{\left((8k + 4)j\right)! \  j!}{(2j)!\left((4k + 1)j\right)!} 
\end{equation}
Now we note that \[
   \ \binom{(4k+2)j-1/2}{(4k+1)j}= \frac{\left( (8k+4)j-1\right)!!}{(2j-1)!! \ 2^{4k+1} \ \left( (4k+1)j \right)! }
 \]
Using this equality  combined with \eqref{eq:bifactorial} for  $l=j$, we simplify \eqref{eq:coefficients1} as:
\begin{equation} \label{eq:true_equality} 2^{(4k+2)j}  \left((4k+2)j\right)! \left( (8k+4)j-1\right)!!= \left((8k + 4)j\right)! \end{equation}
Since  \eqref{eq:true_equality} manifestly holds (it is \eqref{eq:bifactorial} for $l=(4k+2)j$), the result is proved.
\end{proof}


\section{Relations to the work of Beukers, Cohen, Mellit}
\label{sec:relation}

In this Section we show how our Equation~\eqref{eq:2a} arises from the work \cite{MR3613122} as a mirror of the wrong dimension.

\subsection{Finite hypergeometric functions and point counting}
\label{sec:finitehypergeometric}

We summarise the main result in \cite{MR3613122}.
\smallskip

Let $v,w \in \QQ^d$ be such that $\forall i,j \in \{1, \dots, d\}$ $v_i \neq w_j \ \mathrm{mod} \ \ZZ$ and the polynomials $\prod_{j=1}^d (x-e^{2\pi i v_j})$ and $\prod_{j=1}^d (x-e^{2\pi i w_j})$ are products of cyclotomic polynomials. Then there exist natural numbers $p_1, \dots, p_r$ and $q_1, \dots, q_s,$ with $p_1 + \dots + p_s= q_1 + \dots + q_s,$ such that
$$ \prod_{j=1}^d \frac{x-e^{2\pi i v_j}}{x-e^{2\pi i w_j}}= \frac{\prod_{j=1}^r x^{p_j}-1}{ \prod_{j=1}^s x^{q_j}-1}$$ and the analytic hypergeometric function 
\[
{}_{d}F_{d-1}(v, w | \lambda)=\sum_{n=0}^{\infty} \frac{(v_1)_n \cdots (v_d)_n}{(w_1)_n \cdots (w_d)_n} \  \lambda^n
\quad \text{where} \; (x)_n=
\begin{cases}
x (x+1) \cdots (x+n-1) \; &\text{if} \; n\geq 1\\
1  \; & \text{if} \; n=0
\end{cases}
\]
can be rewritten in the form:
\begin{equation} \label{eq:hypergeo_2nd_form}{}_{d}F_{d-1}(v, w | \lambda)= \sum_{j=0}^{\infty} \frac{(p_1 j)! \dots (p_r j)!}{(q_1 j)! \dots (q_s j)! }  \ M^{-j} \lambda^j, \quad M=\frac{ p_1^{p_ 1} \dots p_r^{p_r}}{ q_1^{q_ 1} \dots q_r^{q_r}}
\end{equation} 

Similarly, the  finite hypergeometric function $H_q(v,w|\lambda)$ --- where $q$ is a prime power coprime with $\hcf (v,w)$ --- can be written in terms of $p_i,q_j$ only \cite[Definition 1.1, Theorem 1.3]{MR3613122}.
\smallskip

For all $\alpha \in \mathbb{F}_q^{\times},$ Beukers, Cohen and Mellit consider the quasiprojective (in fact affine) variety $W_{\alpha}$ given by the homogeneous equations:
\begin{displaymath}  
  \left\{ \begin{array}{ll}
            &y_1+y_2 + \dots +y_r- x_1- \dots - x_s=0 \\
            \vspace{0.02cm}
            &\alpha \cdot y_1^{p_1} \dots y_r^{p_r}= x_1^{q_1} \dots  x_s^{q_s}
\end{array} \right.  \ \ \big(\text{for all}\; j,\;x_j,y_j \neq 0 \big)  
\end{displaymath} 
and prove the following, \emph{see}  \cite[Theorem 1.5]{MR3613122} for the precise statement and \cite[Section 5]{MR3613122} for its proof:

\begin{thm} \label{thm:BCM} If $\mathrm{gcd}(p_1, \dots, p_r, q_1, \dots, q_s)=1$ and $M\cdot \alpha \neq 1,$ there exists a nonsingular completion $\overline{W}_{\alpha}$ of $W_{\alpha}$ such that $|\overline{W}_{\alpha}(\mathbb{F}_q)|=H_q(v,w,M \cdot  \alpha)$ up to factors depending only on $p_i,q_j.$  \end{thm}

\subsection{The family of \mbox{$3$-folds} of Theorem~\ref{thm:2a}}
\label{sec:family-3-folds}

We provide some additional context for our statement of Theorem~\ref{thm:2a}. Specifically, we build a precise connection with Theorem~\ref{thm:BCM}.
\smallskip

Fix $k >0$ integer. We specialise the above discussion to the pair of vectors $v_k,w_k$ in $\QQ^{d_k}, \ d_k=6k+2$:
\begin{equation*}  v_k=\left( \frac{j}{8k+4}\right)_{{ j \in \{1, \dots, 8k+3 \} \setminus \left( \{4i\}_{i\in \{1, \dots, 2k\}} \cup \{4k+2\} \right)}}  \quad   w_k=\left(\frac{l}{2k+1},  \frac{m}{4k+1}\right)_{ l \in \{1, \dots, 2k+1\}, \ m \in \{1, \dots, 4k+1\}} 
\end{equation*}
This leads to
\begin{equation*}  p_k=(8k+4,1) \ \ q_k=(2,2k+1,2k+1,4k+1)
  \end{equation*}
For $\lambda=M_k \cdot \alpha$, by \eqref{eq:hypergeo_2nd_form}
  \begin{equation*}{}_{d_k}F_{d_k-1}(v_k, w_k | \lambda)= \sum_j \frac{((8k + 4)j)! j!}{(2j)!{((2k + 1)j)!}^2 ((4k + 1)j)!} M_k^{-j} \lambda ^j
 = \widehat{I}_k(\alpha) \end{equation*}  
that is, the regularised $I$-function of Equation~\eqref{eq:I-function}.

Still following the above discussion, for all $\alpha$ in $\CC^{\times}$ consider the quasiprojective (in fact, affine) \mbox{$3$-fold} $W_{k,\,\alpha}$ given by the homogeneous equations:
\begin{equation}   \label{eq:BCM_ours}
  \left\{ \begin{array}{ll}
            &y_1+y_2- (x_1+x_2+x_3+x_4)=0 \\
            \vspace{0.02cm}
            
&\alpha \cdot y_1^{8k+4}y_2= x_1^2 x_2^{2k+1} x_3^{2k+1} x_4^{4k+1}
\end{array} \right. \ \ \big(\text{for all}\; j,\;x_j,y_j \neq 0 \big)  
\end{equation}

Next, we manipulate the Equations~\eqref{eq:BCM_ours} to rewrite them as in the Introduction. By solving the first Equation for $y_2$ and writing $x_0=y_1$,  for all $\alpha$ in $\CC^{\times}$ the system of Equations~\eqref{eq:BCM_ours}  leads to the Equation of a \mbox{$3$-fold} in $\PP^4$:
\[
\alpha \cdot {x_0}^{8k+4}(x_1+x_2+x_3+x_4-x_0)={x_1}^2{x_2}^{2k+1}{x_3}^{2k+1}{x_4}^{4k+1}
\]
The \mbox{$3$-fold} $W_{k,\, \alpha} $ is the intersection of the projective \mbox{$3$-fold} defined by the Equation above with the torus 
$\mathbb{T}^4 \subset \PP^4$ and, in the affine chart $ \{x_0\neq 0\}$ with coordinates $u_i=x_i/x_0$, it is described by the Equation:
\begin{equation*}
\alpha \cdot (u_1+u_2+u_3+u_4-1) - {u_1}^2{u_2}^{2k+1}{u_3}^{2k+1}{u_4}^{4k+1}=0 \quad 
\left(\forall i \ \   u_i \neq 0  \right)
\end{equation*}
This is the same as Equation~\eqref{eq:2a}; hence the meaning of the symbol $W_{k,\alpha}$ is the same as in the Introduction.

\smallskip

Theorem~\ref{thm:BCM} computes $|\overline{W}_{k, \alpha}(\FF_q)|$ in terms of the finite analog of $\widehat{I}_k$. The statement can be interpreted as asserting that the manifold $W_{k}$ of Equation~\eqref{eq:2a} is a sort of LG mirror of the family of del Pezzo surfaces $X_{8k+4} \subset \PP(2,2k+1,2k+1, 4k+1)$ of the wrong dimension. Indeed, by~\cite{MR873655} the only interesting cohomology group of $\overline{W}_{k, \alpha}$ occurs in degree $3$, hence by the Weil conjectures the main contribution to $|\overline{W}_{k, \alpha}(\FF_q)|$ comes from $H^3( \overline{W}_{k,\alpha}, \QQ)$. Presumably for this completion we also have $H^3( \overline{W}_{k,\alpha}, \QQ)=\mathrm{gr}^W_3 H^3(W_{k,\alpha}, \QQ)$ --- as we do in Lemma~\ref{lem:ptcpt}.


\section{Proof of Theorem~\ref{thm:2a}}
\label{sec:proof}

In this Section we prove Theorem~\ref{thm:2a}, stated in the Introduction.
\smallskip

Fix $k>0$ integer and let $U_k=\CC^\times \setminus \{\alpha_{k,0}\}$,  as in the Introduction. We prove that for all $\alpha \in U_k$ there is an isomorphism of pure Hodge structures:
\begin{equation} \label{eq:stalks}
\mathrm{gr}_3^W H^3_c (W_{k,\alpha},  \QQ) (1)= H^1(Y_{k,\alpha}, \QQ)
\end{equation}

\paragraph{Notation}
For convenience in what follows we set $a=1/\alpha$. We write Equation \eqref{eq:2a} as
\begin{equation}  \label{eq:3-fold-a} 
\left( u_1+u_2+u_3+u_4-1 -a {u_1}^2{u_2}^{2k+1}{u_3}^{2k+1}{u_4}^{4k+1}=0 \right) \quad \subset \TT^4 \times \CC^\times
\end{equation} 
and Equation \eqref{eq:1a} as
\begin{equation} \label{eq:curve-a} \left(a^2y^2-a^3h_{k,\frac{1}{a}}(t_0,t_1)=0 \right) \subset  \PP(1,1,3k+2) \times \CC^{\times}
\end{equation} where $h_{k,\alpha}$ is as in Equation \eqref{eq:1b}.  We denote by  $W_{k,a}$  and $Y_{k,a}$ the fibres of \eqref{eq:3-fold-a} and \eqref{eq:curve-a} over $a$, so $W_{k, a}=W_{k,\alpha}$ and $Y_{k, a}=Y_{k,\alpha}$ where $\alpha=1/a$.
\smallskip

For the rest of this Section we fix $k >0$ and $a \neq a_{k,0}$, where
$a_{k, 0}=1/\alpha_{k,0}$, and we we omit all reference to $k$ and
$a$. Also, we write
$p_{\CC^{\times}} \colon Y_{\CC^{\times}}=p^{-1}(\CC^{\times}) \to
\CC^{\times}$
for the restriction of the $2:1$ cover $p \colon Y \to \PP^{1}$.

\paragraph{Proof of Theorem~\ref{thm:2a}} Let $W\subset \widehat{W}$
be the partial compactification of Equation~\eqref{eq:widehat_W} and
$\phi\colon \widehat{W} \to \CC^\times$ the projective degree~2 del
Pezzo fibration of Lemma~\ref{lem:dPstructure}.
Applying in sequence Lemmas~\ref{lem:ptcpt}, \ref{lem:onthebase}, \ref{lem:final} we see:
\[
\gr^W_3 H^3_c(W, \QQ)=\gr^W_3 H^3_c(\widehat{W}, \QQ)=\gr^W_3 H^1_c(\CC^\times, R^2\phi_\star \QQ)=
\left(\gr^W_1 H^1_c (\CC^\times, p_{\CC^\times} \QQ)\right)(-1)
\]
and the latter group is $H^1(Y, \QQ)(-1)$. \qed

\paragraph{The partial compactification} We construct a partial
compactification
$\widehat{W}\subset \PP(1,1,2,2)/_{\mu_2} \times \CC^\times$ of $W$
(where $\mu_2$ acts on $\PP(1,1,2,2)$ as described
below). Lemma~\ref{lem:ptcpt} states that
$\mathrm{gr}_3^W H^3_c(W, \QQ)=\mathrm{gr}_3^W H^3_c(\widehat{W},
\QQ)$
hence for our purpose we might as well work with $\widehat{W}$ in
place of $W$. The advantage of working with $\widehat{W}$ is that, as
stated in Lemma~\ref{lem:dPstructure}, the second projection
$\phi \colon \widehat{W}\to \CC^\times$ is a fibration with fibres del
Pezzo surfaces, and the point of the substitution~\eqref{eq:change} is
precisely to make this structure manifest. (In
Appendix~\ref{sec:appendix} we explain how we discovered the del Pezzo
fibration structure, and the substitution, by the methods of the
minimal model program.)

\smallskip

Consider the weighted projective space $\PP(1,1,2,2)$ with weighted homogeneous coordinates $x_1,x_2,y_1,y_2$, and the quotient $\PP(1,1,2,2)/ \mu_2$ where the group $\mu_2$ acts on the affine coordinates $x=x_2/x_1, y=y_1/x_1^2, z=y_2/x_1^2$ by
\begin{equation} \label{eq:action} (x,y,z) \mapsto (-x,-y,-z)
\end{equation}
Here and in what follows we write
\begin{equation} \label{eq:G}
G= \PP(1,1,2,2)/ \mu_2 \times \CC^{\times}
\end{equation}
and we denote by $t$ the coordinate on $\CC^{\times}$. Note that $G$ is a (noncompact) toric variety and the 4-dimensional torus $\TT_G\subset G$ is the locus $\left(x,y,z,t\neq 0 \right)$ in the affine piece $x_1\neq 0$.
The substitution
\begin{equation} 
 \label{eq:change}  
 u_1=x^{-1} y \quad u_2=x^3z^{-1}t \quad u_3=x^{-1}z^{-1} \quad u_4=x^{-1}z
\end{equation} 
identifies $W$ with 
\begin{equation} 
  \label{eq:W_new_equation}  
 \left( -at^{2k+1} y^2 +yz +z^2 +1 -xz +tx^4=0\right)/\mu_2 \subset \TT_G
\end{equation}
We denote by $\widehat{W}$ the closure of $W$ in $G$, given by the weighted homogeneous equation:
\begin{equation} 
 \label{eq:widehat_W} 
  \left( -at^{2k+1} y_1^2 +y_1y_2 +y_2^2 +x_1^4 -x_1x_2y_2 +tx_2^4=0 \right) / \mu_2 \subset G
\end{equation}
The \mbox{$3$-fold} $\widehat{W}$ is a partial compactification of $W$. 

\begin{rem}
  The \mbox{$3$-fold} $\widehat{W}$ is quasismooth and we think at it as a
  smooth orbifold.  More precisely $\widehat{W}$ has nonisolated
  quotient singularities: although it is singular, it is a rational
  homology manifold.  Because of this, for the purpose of
  cohomological computations, we can pretend that $\widehat{W}$ is
  smooth. Below we take the convention that the set of regular values
  of the map $\phi \colon \widehat{W} \to \CC^\times$ is the set of
  values $t \in \CC^\times$ such that the corresponding fibre
  $\widehat{W}_t$ is quasismooth.
\end{rem}

\begin{lem} 
 \label{lem:ptcpt}  
  There is an identity of pure Hodge structures:
  \begin{equation} \label{eq:gr_3}
    \mathrm{gr}_3^W H^3_c(W, \QQ)=\mathrm{gr}_3^W H^3_c(\widehat{W}, \QQ)
  \end{equation}  
  \end{lem}
\begin{proof}
Note first that the \mbox{$3$-fold} $W$ is nonsingular but noncompact, thus $H^3_c (W, \QQ)$ is a mixed Hodge structure with weights $\leq 3$, and so is $H^3_c (\widehat{W}, \QQ)$. Consider the divisor $D= \widehat{W} \setminus W$, and denote by $i\colon D \hookrightarrow \widehat{W}$  and $j\colon W  \hookrightarrow  \widehat{W}$ the inclusions. We have a  long exact sequence of mixed Hodge structures
    \begin{equation*} 
     \cdots \to H^2_c(D, \QQ)  \to H^3_c(W,\QQ) \to H^3_c(\widehat{W}, \QQ) \to H^3_c(D,\QQ)  \to \cdots 
\end{equation*}
To prove~\eqref{eq:gr_3}, we check that $\mathrm{gr}_3^W H^2_c(D, \QQ)=\mathrm{gr}_3^W  H^3_c(D,\QQ)=(0)$.
To this end, we study the geometry of the surface $D$; $D$ is the union $
D=\bigcup_{i=1}^4 \  D_i$, where: 
\[
D_1=\widehat{W}\cap \left (x_1=0\right), \quad  
D_2=\widehat{W}\cap \left (x_2=0\right), \quad 
D_3=\widehat{W}\cap \left (y_1=0\right), \quad 
D_4=\widehat{W}\cap \left (y_2=0\right)
\]
One can check that $D \subset \widehat{W}$ is (locally the quotient of) a simple normal crossing divisor with no $0$-dimensional stata. By setting 
\[
D^{[1]}= \bigsqcup_{i} D_{i} \quad \text{and} \quad D^{[2]}= \bigsqcup_{i < j} D_{ij}, \quad \text{with} \ D_{ij}= D_i \cap D_j
\] 
we get a strict simplicial resolution $D^{[2]} \rightrightarrows D^{[1]} \rightarrow D$ and the long exact sequence:
\begin{equation} 
\phantomsection \label{eq:ssr_sequence} 
 \cdots \to  \bigoplus\limits_i H^{m-1}_c(D_i, \QQ)   \to \bigoplus\limits_{i <j} H^{m-1}_c(D_{ij}, \QQ)  \to H^m_c(D, \QQ)   \to    \bigoplus\limits_i H^{m}_c(D_i, \QQ) \to \cdots 
\end{equation} 
It follows that $H^2_c(D, \QQ)$ has weights $\leq 2.$ Now choose $m=3$ in \eqref{eq:ssr_sequence} and examine $H^3_c(D, \QQ)$. 
On the left hand side,  $\bigoplus_{i <j} H^{2}_c(D_{ij}, \QQ)$ has weights $\leq 2.$ On the right hand side, $\bigoplus_i H^{3}_c(D_i, \QQ)$ a priori has weights $\leq 3$. To conclude, we next show that, in fact, for $i=1,2,3,4,$ $H^{3}_c(D_i, \QQ)$ has weights $<3$. Consider first the surface $D_1$, given by
\begin{equation*} 
\left( -at^{2k+1} y_1^2 +y_2^2 +y_1y_2  +tx_2^4 =0\right)/ \mu_2  \subset \PP(1,2,2) / \mu_2 \times \CC^{\times}
\end{equation*}
This is the same as the surface 
\begin{equation} 
\label{eq:D_1new} 
\left( -at^{2k+1} y_1^2 +y_2^2 +y_1y_2  +tz_2^2 =0 \right)/\mu_2 \subset  \PP^2 / \mu_2\times \CC^\times 
\end{equation} 
where $z_2,y_1,y_2$ are homogeneous coordinates of $\PP^2$ and  $\mu_2$ acts as $(y_1,y_2)\mapsto (-y_1,-y_2)$ on the affine piece $\left(z_2=1\right)$. Note that the quotient  of  $\PP^2$ by the  $\mu_2$-action is the weighted projective space $\PP(1,1,2)$ with homogeneous coordinates $y_1,y_2, w_2=z_2^2$. In $\PP(1,1,2)$~\eqref{eq:D_1new} becomes
\begin{equation*} 
\left(  at^{2k+1} y_1^2 -y_2^2 -y_1y_2 = tw_2 \right) \subset \PP(1,1,2) 
\end{equation*}  
thus, since $t \neq 0$,  we conclude that $D_1 \simeq \PP^1 \times \CC^{\times}$ and then $H^{3}_c(D_1, \QQ)$ is a pure Hodge structure of weight $2$.
An almost identical argument holds for $D_2.$
The surface $D_4$ is given by
\begin{equation*} \left(-at^{2k+1} y_1^2 +x_1^4+tx_2^4 =0\right)/\mu_2 \subset \PP(1,1,2) / \mu_2 \times \CC^{\times}
\end{equation*} which  is the same as the surface
\begin{equation*} \left(-at^{2k+1} w_1 +z_1+tz_2 =0\right) \subset \PP^2 \times \CC^{\times}
\end{equation*} where $z_1,z_2,w_1$ are homogeoneous coordinates on $\PP^2,$ thus also $D_4 \simeq \PP^1 \times \CC^{\times}$ and $H^{3}_c(D_4, \QQ)$ is a pure Hodge structure of weight $2$. To conclude, consider the surface $D_3$, given by
\begin{equation*}
  \left(y_2^2 +x_1^4 -x_1x_2y_2 +tx_2^4 =0\right)/ \mu_2  \subset \PP(1,1,2) / \mu_2 \times \CC^{\times}
\end{equation*}
By means of the substitution $y_2\mapsto y_2 - x_1x_2 /2$ write this as
\begin{equation*}
  \left(y_2^2 +x_1^4 -\frac{x_1^2x_2^2}{4}+tx_2^4 =0\right)/ \mu_2  \subset \PP(1,1,2) / \mu_2 \times \CC^{\times}
\end{equation*}
and note  that this is the same as the surface
\begin{equation} \label{2:1coverD_3}
  \left(w_2^2 =z_1z_2 ( -z_1^2 +\frac{z_1z_2}{4}-tz_2^2 ) \right) \subset \PP(1,1,2) \times \CC^{\times}
\end{equation}
where $z_1=x_1^2,z_2=x_2^2,w_2=y_2^2$ are homogeneous coordinates on $\PP(1,1,2)$. A natural compatification of \eqref{2:1coverD_3} is the surface $\overline{D}_3$ given by
\begin{equation*} 
  \left(w_2^2 =z_1z_2(-t_0^2z_1^2 +t_0^2\frac{z_1z_2}{4}-t_0t_1z_2^2) \right) \subset \PP(1,1,2) \times \PP^1
\end{equation*}
Note that  $\overline{D}_3$ is a $2:1$ cover of $\PP^1\times \PP^1$ branched along a divisor in $|\mathcal{O}_{\PP^1 \times \PP^1}(4,2)|$, thus $\overline{D}_3$ is a rational surface and hence $H^3(\overline{D}_3, \QQ)=\left(0\right)$. Then $ H^3({D}_3, \QQ)$  has weights $\leq 2$, since, by setting $\Gamma=\overline{D}_3 \setminus D_3$, we have the long exact sequence of mixed Hodge structures
\begin{equation*}  
     \cdots \to H^{2}_c(\Gamma, \QQ)   \to  H^{3}_c(D_3, \QQ)  \to H^3(\overline{D}_3, \QQ)   \to    \cdots 
\end{equation*}  and $\mathrm{gr}_3^W H^2({\Gamma}, \QQ) =\left( 0 \right)$.
\end{proof}

\begin{lem}
  \label{lem:dPstructure}
  Let $\phi\colon \widehat{W} \rightarrow \CC^{\times}$ be the
  projection onto the second factor. Denote by $\Delta$ the set of
  critical values of $\phi$ and write
  $\Omega=\CC^\times \setminus \Delta$ for the set of regular
  values. Let $j \colon \Omega \hookrightarrow \CC^\times$ be the
  natural inclusion and denote by
  $\phi_{\Omega}: \widehat{W}_\Omega=\phi^{-1}(\Omega) \to \Omega$ the
  induced morphism. Let $\delta_1, \delta_2, \delta \in \CC[t]$ be the
  polynomials:
\begin{equation} \label{eq:delta_12}
    \delta_1(t)=4at^{2k+1}+1 \qquad \delta_2(t)=a^2t^{4k+2} -4t\left(4at^{2k+1}+1\right)^2  \qquad \delta= \delta_1 \cdot \delta_2
\end{equation} Write $\Omega_1=\CC^\times \setminus \{\delta_1=0\}$,
$\widehat{W}_{\Omega_1}=\phi^{-1}(\Omega_1)$. Denote by $K$ the
function field $k(\CC^\times)=\CC(t)$. Then:
\begin{enumerate}[(1)]
\item $\Delta = \{ \delta =0\}$. If $t$ is a root of $\delta_1$, $\widehat{W}_t$ has a unique non quasismooth point $p_t=(0:0:-2:1)$; if $t$ is a root of $\delta_2$, $\widehat{W}_t$ has a unique non quasismooth point
\begin{equation} \label{eq:q_t}
  q_t= \left(1:\sqrt{\frac{2 \delta_1(t)}{at^{2k+1}}}:\frac{1}{\delta_1(t)}: \frac{2at^{2k+1}}{\delta_1(t)}\right)
\end{equation}
In both cases the non quasismooth point is an ordinary double point. 
\item After the change of coordinates:
\begin{equation} \label{eq:newy}
  y_1 \mapsto \frac{y_1}{2}-\frac{x_1x_2}{2(4at^{2k+1}+1)} \quad
  \text{and} \quad y_2  
\mapsto \frac{y_1}{2}+y_2-\frac{x_1x_2}{2}
\end{equation} 
the equation of the fibre $\widehat{W}_t$ over $t \in \Omega_1$ is:
\begin{equation} \label{eq:notdelta_1}
\left( -(4at^{2k+1} +1)y_1^2 +y_2^2 +x_1^4 +tx_2^4-\frac{at^{2k+1}}{4at^{2k+1}+1} (x_1x_2)^2=0 \right)/\mu_2 \subset \PP(1,1,2,2) /\mu_2
\end{equation}   
\item
  For all $t \in \Omega$ the fibre  $\widehat{W}_t$ is a quasismooth del Pezzo surface of degree $2$ with two $1/4 \cdot (1,-1)$ points $p^{+}_t, p^{-}_t$, on $(x_1=x_2=0)$, and intersecting $(y_1=y_2=0)$ in two points $q^{+}_t$ and $q^{-}_t$.  In the coordinates of Equation \eqref{eq:notdelta_1},
  \begin{equation}  \label{eq:pm}  p_t^{\pm}=\left( 0:0:1: \pm \frac{\sqrt{\delta_1(t)}}{2} \right) 
  \quad \text{and}  \quad q_t^\pm=\left(\sqrt{\frac{at^{2k+1}\pm \sqrt{\delta_2(t)}}{2 \cdot \delta_1(t)}}:1:0:0 \right)
  \end{equation}
\item For all $t \in \Omega$ the fibre $\widehat{W}_t$ has Picard rank $r=h^2(\widehat{W}_t,\QQ)=2$. More specifically, $\widehat{W}_t$ contains a configuration of lines as pictured in Figure~\ref{fig:FibrePic}, and a basis of $\Pic (\widehat{W}_t)_{\QQ}$ is given the classes of the curves $C_{t,1}$ and $C_{t,2}$.
\item In the variables of Equation \eqref{eq:notdelta_1}, the
  restriction $\phi_1 \colon Y_1=\widehat{W} \cap  \left(
    x_1=x_2=0\right)  \to \CC^\times$ is a $2:1$ branched cover with branch locus $\{\delta_1=0\}$, and the restriction $\phi_2 \colon Y_2=\widehat{W}_{\Omega_1}  \cap  \left( y_1=y_2=0\right)  \to \Omega_1$ is a $2:1$ branched cover with branch locus $\{\delta_2=0\}$. In particular,
\begin{equation} \label{eq:preimage}
    \phi_1^{-1}(t)= \left\{ \begin{array}{ll}
p^{\pm}_t & \textrm{if $\delta_1(t) \neq 0$}\\
p_t & \textrm{if $\delta_1(t) = 0$}
                            \end{array} \right.
      \quad \text{and} \quad
\phi_2^{-1}(t)=\left\{ \begin{array}{ll}
q^{\pm}_t & \textrm{if $\delta_2(t) \neq 0$}\\
q_t & \textrm{if $\delta_2(t) = 0$}
 \end{array} \right.
   \end{equation}
                                             
\item The Picard rank of $\widehat{W}_{K}$ is $\rho=1$.
\item $R^3\phi_{\star}\QQ = R^1\phi_{\star} \QQ=(0)$. 
\item The natural homomorphism $
R^2 \phi_\star \QQ\to j_\star j^\star R^2 \phi_\star \QQ =  j_\star R^2 \phi_{\Omega\, \star} \QQ 
$
is an isomorphism.
 \end{enumerate}
\end{lem}

 \begin{figure}[!ht] \centering
   \begin{tikzpicture}[ scale=1] 
     \clip(-4,-3) rectangle (4,3);
  \draw[fill, ultra thick] (0,2) circle [radius=0.08]; 
  \draw[fill, ultra thick] (0,-2) circle [radius=0.08]; 
  \draw[] (-2,0) circle [radius=0.05]; 
  \draw[] (2,0) circle [radius=0.05]; 
  \node at (0.1,2.7) {$p_t^{+}$};
  \node at (0.1,-2.7) {$p_t^{-}$};
  \node at (2.7, 0) {$q_t^{+}$};
  \node at (-2.7,0) {$q_t^{-}$};
  \draw[thick] (-2.5,-0.5) -- (0.5,2.5); 
  \draw[thick]  (-0.5,2.5) -- (2.5,-0.5); 
  \draw[thick]  (-2.5,0.5) -- (0.5,-2.5); 
  \draw[thick](-0.5,-2.5) -- (2.5,0.5); 
   \draw[] (-0.8, 0.8) -- (-1.2,1.2);
   \draw[] (0.8, -0.8) -- (1.2,-1.2);
   \draw[] (-0.8, -0.8) -- (-1.2,-1.2);
  \draw[] (-0.8, -0.6) -- (-1.2,-1);
   \draw[] (0.8, 0.8) -- (1.2,1.2);
   \draw[] (0.8, 0.6) -- (1.2, 1);
   \node at  (-1.5,1.5) {$C_{t,1}$};
    \node at  (1.5,1.5) {$C_{t,2}$};
     \node at  (-1.5,-1.5) {$C_{t,3}$};
      \node at  (1.5,-1.5) {$C_{t,4}$};
 \end{tikzpicture}
 \caption{A quasismooth fibre $\widehat{W}_t$ and the four points $p_t^{\pm} , q_t^{\pm}$ in $\widehat{W}_t$. Numerically equivalent curves on $\widehat{W}_t$ are marked with the same symbol.} \label{fig:FibrePic}
\end{figure}
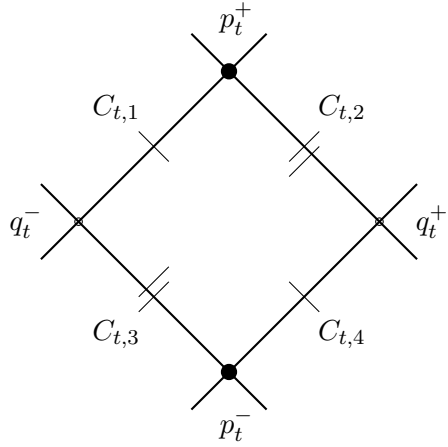

\begin{proof}
To prove~(1), fix $t \in \CC^\times$ and compute the Jacobian of the polynomial in Equation \eqref{eq:widehat_W}. To prove that for all $t \in \Delta$ the non quasismooth point of $\widehat{W}_t$ is an ordinary double point, one can check that the Hessian of \eqref{eq:widehat_W} at that point is invertible.  

Assertion~(2) is a simple substitution.  

To prove~(3), fix  $t \in \Omega$ and consider the quotient map $\sigma \colon \PP(1,1,2,2) \to \PP(1,1,2,2)/\mu_2$. Let $\widehat{V}_t= \sigma^{-1}(\widehat{W}_t)$ be the preimage. Note that:
\[
 -K_{\widehat{V}_t}=\sigma^{\star}\left(-K_{\widehat{W}_t}\right) \quad \text{and} \quad {K^2_{\widehat{V}_t}}=2 \cdot  {K^2_{\widehat{W}_t}}
\]
since $\sigma$ is $2:1$ and et\'ale in codimension 1. We have $-K_{\widehat{V}_t}=\mathcal{O}(2)$  and $ -K_{\widehat{V}_t}^2=4$, thus $\widehat{V}_t \subset \PP(1,1,2,2)$ is a del Pezzo surface of degree $4$. The orbifold points of $\widehat{V}_t$, of type $1/2 \cdot (1,1)$, are the two points $ P_t^{\pm}$ of $\widehat{V}_t$ on the line $\left(x_1=x_2=0\right)$; in the coordinates of Equation \eqref{eq:notdelta_1},   
\[
  P_t^{\pm}= \left(0:0:1: \pm \frac{\sqrt{\delta_1(t)}}{2}\right)
\]
Note that, by \eqref{eq:action}, on the affine piece $(y_1=1)$ the group $\mu_2$ acts by \[(x_1,x_2,y_2) \to (ix_1,-ix_2,y_2)\] Hence $\widehat{W}_t \subset \PP(1,1,2,2)/\mu_2$ is a del Pezzo surface of degree $2$ with two $1/4 \cdot (1,-1)$ points  $p^{\pm}_t=\sigma(P_t^{\pm})$, as in \eqref{eq:pm}; these are the only orbifold points of $\widehat{W}_t$ since $(1:0:0:0)$ and $(0:1:0:0)$ do not satisfy \eqref{eq:notdelta_1}. Setting $y_1=y_2=0$ in Equation \eqref{eq:notdelta_1}, one finds that $\widehat{W}_t \cap \left( y_1=y_2=0 \right)$ is given by the two points $q_t^{\pm}$ in \eqref{eq:pm}.

To prove~(4), note that the crepant resolution  $\widetilde{W}_t$ of $\widehat{W}_t$ is a smooth weak del Pezzo surface of degree $2$. By Demazure's Theorem this is the blow-up of $\PP^2$ in $7$ nongeneral points, thus it has Picard rank $r=8$. Then the surface $\widehat{W}_t$, obtained by blowing down $6$ exceptional curves on $\widetilde{W}_t$, has Picard rank $r=2$. 
We explain how the geometry of $\widehat{W}_t$ singles out a distinguished basis of generators of ${\mathrm{Pic}(W_t)}_{\QQ}$. Setting \[\nu_{\pm}(t)= \sqrt{\frac{at^{2k+1} \pm \sqrt{\delta_2(t)}}{ 2 \cdot \delta_1(t)}}\] on $\PP(1,1,2,2)$ Equation \eqref{eq:notdelta_1} factors as:
  \[(y_2 - \sqrt{\delta_1(t)}y_1) (y_2 +\sqrt{\delta_1(t)}y_1) +\left(x_1 -\nu_{+}(t)  x_2\right)\left(x_1 + \nu_{+}(t)  x_2\right)\left(x_1 - \nu_{-}(t)  x_2\right)\left(x_1 +\nu_{-}(t)  x_2\right)=0
  \]
This exhibits eight lines on $\widehat{V}_t$, four of which passing through $P_t^{+}$, the other four passing thorugh $P_t^{-}$, as pictured on the left of Figure~\ref{fig:3surfaces}.
The four lines on $\widehat{V}_t$ through $P_t^{+}$ correspond to two orbits under the $\mu_2$ action, and the same holds for the four lines through $P_t^{-}$. Namely, in $\PP(1,1,2,2) /\mu_2$  Equation \eqref{eq:notdelta_1} can only be factored as:
\[
  (y_2 - \sqrt{\delta_1(t)}y_1)(y_2 +\sqrt{\delta_1(t)}y_1) +\left(x^2_1 -\nu_{+}^2(t)  x^2_2\right)\left(x^2_1 - \nu_{-}^2(t)  x^2_2\right)=0
\]
This exhibits four lines on the surface $\widehat{W}_t$, as pictured in the middle of Figure~\ref{fig:3surfaces}, two of which passing through $p_t^{+}:$
\begin{equation} \label{eq:above}
 C_{t,1}\colon   \left(y_2 - \sqrt{\delta_1(t)}y_1=x^2_1 -\nu_{+}^2(t) x^2_2=0 \right)  \qquad C_{t,2}\colon \left( y_2 - \sqrt{\delta_1(t)}y_1=x^2_1 \nu_{-}^2(t) x^2_2=0\right)
\end{equation}
the other two passing through $p_t^{-}$:
\begin{equation} \label{eq:below}
C_{t,3} \colon \left(y_2 + \sqrt{\delta_1(t)}y_1=x^2_1 -\nu_{+}^2(t) x^2_2=0 \right) \qquad  \ C_{t,4} \colon \left( y_2 + \sqrt{\delta_1(t)}y_1=x^2_1 - \nu_{-}^2(t) x^2_2=0\right)
\end{equation} 
To see that $C_{t,1} \sim C_{t,4}$ note that, on $\widetilde{W}_t$,
the union of the strict transform $\widetilde{C}_{t,1}$ and the three
exceptional curves above $p_t^{+}$, as pictured on the right of
Figure~\ref{fig:3surfaces}, supports a fibre of a conic bundle where 
another fibre is supported on the union of the strict transform
$\widetilde{C}_{t,4}$ and the three exceptional curves above
$p_t^{-}$. Taking direct image this implies indeed that $C_{t,1} \sim
C_{t,4}$. Similarly, $C_{t,2} \sim C_{t,3}$. Then the two curves $C_{t,1}$ and $C_{t,2}$
form a basis of $\mathrm{Pic}(\widehat{W}_t) _{\mathbb{Q}}.$

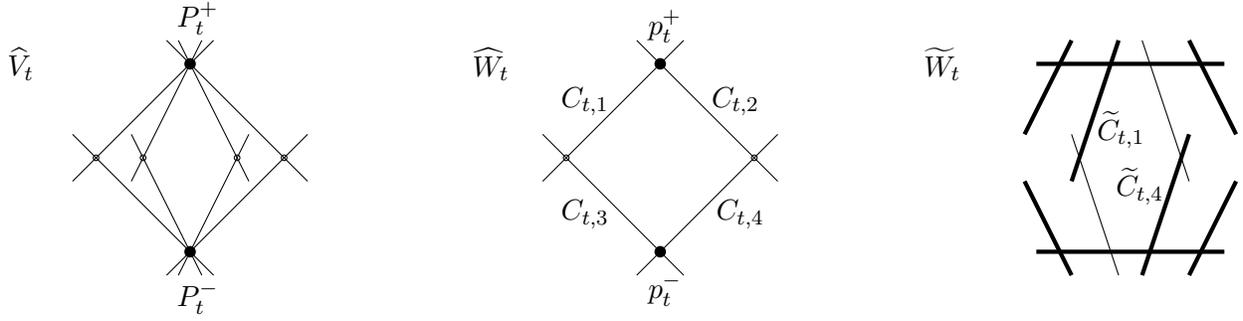
\begin{figure}[!ht] \centering 
\begin{tikzpicture} [scale=1.25]
\clip(-7,-2) rectangle (7,2);
  \draw[fill, thick] (0,1) circle [radius=0.05]; 
  \draw[fill, thick] (0,-1) circle [radius=0.05]; 
  \draw[] (-1,0) circle [radius=0.03]; 
  \draw[] (1,0) circle [radius=0.03]; 
  \node at (0.05,1.4) {$p_t^{+}$};
  \node at (0.05,-1.4) {$p_t^{-}$};
  
  \draw[] (-1.25,-0.25) -- (0.25,1.25); 
  \draw[] (-0.25,1.25) -- (1.25,-0.25); 
  \draw[]  (-1.25,0.25) -- (0.25,-1.25); 
  \draw[](-0.25,-1.25) -- (1.25,0.25); 

  \node at  (-0.8,0.6) {$C_{t,1}$};
    \node at  (0.8,0.6) {$C_{t,2}$};
     \node at  (-0.8,-0.6) {$C_{t,3}$};
      \node at  (0.85,-0.6) {$C_{t,4}$};
      \node at (-1.8, 1) {$\widehat{W}_t$};
      
  \draw[fill, thick] (-5,1) circle [radius=0.05]; 
  \draw[fill,  thick] (-5,-1) circle [radius=0.05]; 
   \draw[] (-4,0) circle [radius=0.03]; 
  \draw[] (-6,0) circle [radius=0.03]; 
  \draw[] (-5.5,0) circle [radius=0.03]; 
  \draw[] (-4.5,0) circle [radius=0.03]; 
   \node at (-4.92,1.49) {$P_t^{+}$};
  \node at (-4.92,-1.49) {$P_t^{-}$};

   \draw[] (-6.25,-0.25) -- (-4.75,1.25); 
  \draw[] (-5.25,1.25) -- (-3.75,-0.25); 
  \draw[]  (-3.75,0.25) -- (-5.25,-1.25); 
  \draw[](-4.75,-1.25) -- (-6.25,0.25); 
  \draw[] (-5.625,-0.25) -- (-4.875,1.25); 
  \draw[]  (-5.125,1.25) -- (-4.375,-0.25); 
  \draw[]  (-4.375,0.25) -- (-5.125,-1.25);  
  \draw[] (-5.625, 0.25) -- (-4.875, -1.25);
   \node at (-6.8,1) {$\widehat{V}_t$};
   
   \draw[ ultra thick] (4,1) -- (6,1); 
   \draw[ultra thick] (4,-1) -- (6,-1); 
   \draw[ ultra thick]  (4.375,1.25) -- (3.875,0.25); 
   \draw[ ultra thick]  (4.375,-1.25) -- (3.875,-0.25);
    \draw[ ultra thick]  (5.625,1.25) -- (6.125,0.25);
   \draw[ ultra thick]  (5.625,-1.25) -- (6.125,-0.25);

     \draw[ ultra thick]  (4.875,1.25) -- (4.375,-0.25); 
   \draw[]  (4.875,-1.25) -- (4.375,0.25);
    \draw[ ]  (5.125,1.25) -- (5.625,-0.25);
   \draw [ ultra thick]  (5.125,-1.25) -- (5.625,0.25);

 \node at (4.9,0.3) {$\widetilde{C}_{t,1}$};
  \node at (5.1,-0.3 ) {$\widetilde{C}_{t,4}$};
   \node at (3,1) {$\widetilde{W}_t$};
   \end{tikzpicture}
  \caption{The surfaces $\widehat{V}_t$, $\widehat{W}_t$ and $\widetilde{W}_t$. The eight lines on $\widehat{V}_t$ descend to four lines on  $\widehat{W}_t$. The two sets of bold lines on $\widetilde{W}_t$ correspond to two distinct fibers of a conic bundle over $\PP^1$.} \label{fig:3surfaces}
\end{figure}

Statement~(5) follows immediately from Equation \eqref{eq:notdelta_1}, and by~(1) and~(3). In particular, to prove ~\eqref{eq:preimage} one can check by \eqref{eq:newy} that, in the variables of Equation \eqref{eq:widehat_W}, if $t$ is a root of $\delta_1$ then $\phi_1^{-1}(t)=(0:0:2:-1)=p_t$, and if t is a root of $\delta_2$ then $\phi_2^{-1}(t)$ is given by
\[ \left(\sqrt{\frac{at^{2k+1}}{2 \delta_1(t)}}: 1  :\sqrt{\frac{at^{2k+1}}{2 {\delta^3_1(t)}}}: \sqrt{\frac{2{(at^{2k+1})}^3}{{\delta^3_1(t)}}}\right)
\]
which is the point $q_t$ in~\eqref{eq:q_t}.
  
We prove~(6) as follows. Note that the function fields of $Y_1$ and $Y_2$ are $k(Y_1)=K(\sqrt{\delta_1})$ and $k(Y_2)=K(\sqrt{\delta_2})$. Set $L=K(\sqrt{\delta_1},\sqrt{ \delta_2})$ and consider the lattice of Galois field extensions:
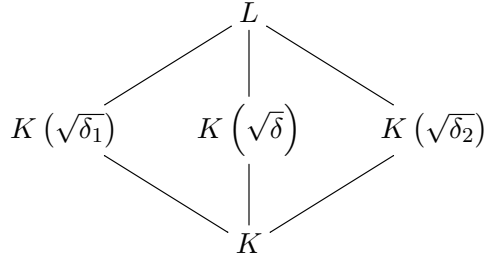
\begin{equation*}
\xymatrix{ &L & \\
K\left(\sqrt{\delta_1}\right)\ar@{-}[ur] & K\left(\sqrt{\delta}\right)\ar@{-}[u] & K\left(\sqrt{\delta_2}\right)\ar@{-}[ul] \\
&K\ar@{-}[ul]\ar@{-}[u]\ar@{-}[ur] & }
\end{equation*}
The Galois group $\mathrm{Gal}(L/K) \simeq C_2 \times C_2$ is generated by the two authomorphisms $\sigma_1$ and $\sigma_2$, where:
\[
  {\sigma_1}_{|K\left(\sqrt{\delta_1}\right)}=\mathrm{id} \quad  \sigma_1(\sqrt{\delta_2})=-\sqrt{\delta_2 }\quad \text{and} \quad   {\sigma_2}_{|K\left(\sqrt{\delta_2}\right)}=\mathrm{id} \quad \sigma_2(\sqrt{\delta_1})=-\sqrt{\delta_1}
\]
By~(4), $\mathrm{Pic}(\widehat{W}_L)$ is generated by the classes of the curves $C_1$ and $C_2$, of Equation \eqref{eq:above}. Then by Galois descent $\mathrm{Pic}(\widehat{W}_K)$ has only one generator, given by the class of the curve $C_1+C_2+C_3+C_4 \subset \widehat{W}_K$, since $C_1+C_2+C_3+C_4$ is the only curve in $\widehat{W}_L$ which is invariant with respect to the induced action of $\mathrm{Gal}(L/K)$.

Assertion~(7) follows from the fact that, by~$(1)$ and~$(4)$, $\forall \ t \in \CC^{\times}$ $H^1(\widehat{W}_t, \QQ)=H^3(\widehat{W}_t, \QQ)=(0)$.

To prove~(8), consider a singular value $s\in \Delta$ and the fibre $\widehat{W}_s$. To check that $R^2 \phi_\star \QQ \to j_\star R^2 \phi_{\Omega\, \star} \QQ$ is an isomorphism in a neighbourhood of $s$, it is enough to show that the natural homomorphism 
\[
H^2(\widehat{W}_s, \QQ) \to H^2(\widehat{W}_t, \QQ)^{T_s}
\]
is an isomorphism, where $t\in \CC^\times$ is near $s$, $T_s\colon H^2(\widehat{W}_t, \QQ) \to  H^2(\widehat{W}_t, \QQ)$ 
is the local monodromy operator at $s$, and $H^2(\widehat{W}_t, \QQ)^{T_s}$ denotes the group of monodromy invariants. Indeed on the one hand, by the proper base change theorem, $H^2(\widehat{W}_s, \QQ)$ is the fibre at $s$ of $R^2\phi_\star \QQ$, and on the other hand $H^2(\widehat{W}_t, \QQ)^{T_s}$ is the fibre at $s$ of $j_\star R^2 \phi_{\Omega \, \star}\QQ$.

We have an exact triangle of constructible complexes on $\widehat{W}_s$:
\[
\QQ\to \psi \, \QQ \to \varphi \, \QQ \xrightarrow{+1}
\]
where $\psi$ and $\varphi$ are the nearby and vanishing cycle functors
\cite[Expos\'e~I]{MR0354656}. Since $\widehat{W}_s$ has isolated
hyperquotient singularities --- in fact by (4) it only has one non
quasismooth point --- $\varphi \, \QQ$ is supported at the non
quasismooth point of $\widehat{W}_s$ and is concentrated in
degree~$2$~\cite{MR0239612}. Thus from the exact sequence:
\[
(0)=H^1(\widehat{W}_s, \varphi \, \QQ) \to H^2 (\widehat{W}_s, \QQ)\to H^2(\widehat{W}_s, \psi \, \QQ) =H^2(\widehat{W}_t,\QQ)\to \cdots
\]
we conclude that the natural homomorphism $H^2(\widehat{W}_s, \QQ) \to H^2(\widehat{W}_t,\QQ)$ is injective. By the local invariant cycle Theorem~\cite[Theorem~1.4.1]{2015arXiv150603642D} then $H^2(\widehat{W}_s, \QQ)$ is the group of monodromy invariant cycles in $H^2(\widehat{W}_t,\QQ)$.
\end{proof}

\begin{lem} \label{lem:onthebase}  There is an identity of  mixed Hodge structures: 
\begin{equation} \label{eq:H3-H1} H^3_c (\widehat{W}, \QQ)= H^1 _c (\CC^{\times}, R^2{\phi}_{\star} \QQ)
\end{equation}
  \end{lem}

\begin{proof}
Note that the two functors  $\phi_{\star}$ and $\phi_{!}$ concide, since the map $\phi$ is proper.
Consider the Leray spectral sequence of the morphism $\phi$ with second page $E_2^{p,q}=H^p_c( R^q\phi_\star \QQ) \Longrightarrow  H_c^{p+q}(\widehat{W}, \QQ)$, which is known to degenerate at the second page.
The groups contributing to $H^3_c(\widehat{W},\QQ)$ are:
\[
H^0_c(\CC^{\times},  R^3 {\phi}_{\star} \QQ), \  H^1_c(\CC^{\times}, R^2 {\phi}_{\star} \QQ), \ H^2_c(\CC^{\times},R^1 {\phi}_{\star} \QQ)  
\]
Identity \eqref{eq:H3-H1} holds if and only if $H^0_c(\CC^{\times},  R^3 {\phi}_{\star} \QQ)= H^2_c(\CC^{\times}, R^1 {\phi}_{\star} \QQ)=(0)$, and indeed this is so, by Lemma~\ref{lem:dPstructure}$(7)$.
\end{proof}

\begin{lem} 
\label{lem:final}
 There is an isomorphism of mixed sheaves on $\CC^{\times}$: 
\begin{equation} 
\label{fasci} 
p_{\CC^{\times} \star} \QQ_{Y_{\CC^{\times}}} \to R^2 \phi_{\star} \QQ_{\widehat{W}}  (1) 
\end{equation} 
\end{lem}

\begin{proof}
The sketch of the proof is as follow: let $\Omega \subset \CC^{\times}$ and $j \colon \Omega \hookrightarrow \CC^{\times}$ be as in Lemma~\ref{lem:dPstructure} and let $p_{\Omega} \colon Y_{\Omega}=p^{-1}(\Omega) \to \Omega$ be the induced morphism; we first construct a homomorphism 
\begin{equation}
  \label{eq:homomorphism}
z \colon p_{\Omega\, \star} \QQ_{Y_\Omega}\to R^2\phi_{\Omega \, \star} \QQ_{\widehat{W}_\Omega} (1)
\end{equation}
and then we show that it is an isomorphism. This concludes the proof since on the one hand it is obvious that $j_\star p_{\Omega\, \star} \QQ_{Y_\Omega} = p_{\CC^{\times} \star} \QQ_{Y_{\CC^{\times}}}$, on the other hand by Lemma~\ref{lem:dPstructure}(8)  $j_\star R^2\phi_{\Omega \, \star} \QQ_{\widehat{W}_\Omega}=R^2\phi_\star \QQ_{\widehat{W}}$.

Below we construct a cycle 
\[
Z\in \CH_2\left(Y_\Omega \times_\Omega \widehat{W}_\Omega\right)
\]
inducing the homomorphism $z$ via the natural maps:
\begin{multline*}
\CH_2 \left(Y_\Omega \times_\Omega \widehat{W}_\Omega\right) \to 
H^{BM}_4 \left(Y_\Omega \times_\Omega \widehat{W}_\Omega\right)
=\\
=\Hom_{D^b_{cc}} \left(Rp_{\Omega\,\star} \QQ_{Y_\Omega},  R\phi_{\Omega\,\star} \QQ_{\widehat{W}_\Omega}[2](1)\right)\to 
\Hom \left(p_{\Omega\,\star} \QQ_{Y_\Omega},  R^2\phi_{\Omega\,\star} \QQ_{\widehat{W}_\Omega}(1)\right)  
\end{multline*}

Consider the diagram
\[
\xymatrix{
   &Y_\Omega \times_\Omega \widehat{W}_\Omega\ar[dr]^{p_\Omega^\prime}\ar[dl] & \\
  Y_\Omega\ar[dr]_{p_\Omega} & & \widehat{W}_\Omega\ar[dl]^{\phi_\Omega}\\
   & \Omega& }
\]
Let $\overline{Z}\subset \widehat{W}_\Omega$ be the union over all $t\in \Omega$ of the four curves $C_{t,1}, \dots, C_{t,4}$ in the fibre $\widehat{W}_t$ described in Lemma~\ref{lem:dPstructure}, that is:
\[
  \overline{Z}= \left(-(4at^{2k+1}+1)y_1^2 +y_2^2= x_1^4 -\frac{at^{2k+1}}{4at^{2k+1}+1} {(x_1x_2)}^2 +tx_2^4=0\right) /\mu_2  \subset \widehat{W}_\Omega
\]
We have that:

\[
Y_\Omega \times_\Omega \overline{Z} = Z_1+Z_2 \subset Y_\Omega \times_\Omega \widehat{W}_\Omega
\]
is the sum of two irreducible components and we take $Z$ to be one of these components. In order to see this, let $K=\CC(t)$ and $L=K(\sqrt{\delta_1}, \sqrt{\delta_2})$, as in the proof of Lemma \ref{lem:dPstructure}, and consider the field extensions $K \subset K(\sqrt{\delta}) \subset L$.
The key point is to notice that $Y_{\Omega}$ is a $2:1$ cover of $\Omega$ with function field $k(Y_\Omega)=K(\sqrt{\delta})$. Indeed, by \eqref{eq:curve-a} $Y$ is a $2:1$ cover of $\PP^1$ branched at the $6k+4$ roots of the polynomial
\[
t \ (4at^{2k+1}+1) \left( -64a^2t^{2k+1}+a^2t^{4k+1}-32at^{2k+1}-4\right) 
\]
By \eqref{eq:delta_12} this is the polynomial $\delta(t)=\delta_1(t) \cdot \delta_2(t)$. Then, since the generator of the Galois group $\mathrm{Gal}(L/K(\sqrt{\delta}))$ exchanges $C_1$ with $C_4$ and $C_2$ with $C_3$ on $\widehat{W}_L$, by Galois descent the cycle $C_1+C_2+C_3+C_4$ on $\widehat{W}_{k(Y_\Omega)}$ splits into two components $C_1+C_4$ and $C_2+C_3$, each defined over $k(Y_\Omega)$ and corresponding to the two irreducible components $Z_1$ and $Z_2$. Note also that $p^\prime_{\Omega\,\star }Z=\overline{Z}$.

It remains to show that the induced homomorphism $z$ is an isomorphism. This can be checked at the generic point $\eta$, or indeed at any point $t\in \Omega$. This is precisely the statement in Lemma~\ref{lem:dPstructure}$(4)$ that the set $\{C_1+C_4, C_2+C_3\}$ is a basis of $\Pic(\widehat{W}_t)_{\QQ}=H^2(\widehat{W}_t,\QQ)$.
\end{proof}


\appendix 
\section{Toric MMP}
\label{sec:appendix}

Let $W$ and $\widehat{W}$ be the fibre of \eqref{eq:3-fold-a} over $a$
and its partial compactification~\eqref{eq:widehat_W} in
$G=\PP(1,1,2,2)/_{\mu_2} \times \CC^{\times}$, as in Section
\eqref{sec:proof}.
\medskip

In this Appendix, we explain how we discovered the del Pezzo fibration
$\phi \colon \widehat{W} \to \CC^{\times}$, and the
substitution~\eqref{eq:change}, by the methods of the minimal model
program for toric hypersurfaces.

The space $\widehat{W}$ has canonical but not terminal singularities. In Section~\ref{sec:conic-bundle-struct} we construct a birational model $W^\prime$ with terminal singularities and a Mori fibration
$\psi \colon W^\prime \to \FF_1$. 

\medskip

As in Section~\eqref{sec:proof}, we fix once and for all an integer
$k > 0$ and a value $a \neq a_{k,0}$, and we omit all reference to
$k$ and $a$ in what follows.
\smallskip

\paragraph{Notation} We set up our notation for toric varieties. For a lattice $L$, we denote by $L_{\RR}=L \otimes_{\ZZ} \RR$ the
associated real vector space. For a torus $\TT $, we denote by
$M=\mathrm{Hom}(\TT, \CC^\times)$ the character lattice and let
$N=\mathrm{Hom}(M, \ZZ)$.

For a fan $\Sigma \subset N_R$, we denote by $F_{\Sigma}$ the
associated toric variety. We denote by $\rho_i$ the primitive
generators of the $1$-dimensional cones of $\Sigma$, and by
$D_i \subset F_{\Sigma}$ the corresponding divisors.

If $F$ is a proper toric variety and $D=\sum a_i D_i$ is a Weil
divisor on $F$, then we denote by
\begin{equation*} \label{eq:polytope_of_D}
P_D=\bigcap_{i} \{ m \colon \langle \rho_i, m \rangle \geq -a_i  \} \subset M_\RR  
\end{equation*}
the polytope of $D$; it is well known that $P_D\cap M$ is a basis of $H^0(F, D)$.

Given a rational polytope $P\subset M_\RR$, we denote by $\Sigma_P$
the normal fan of $P$. If $P$ is not full dimensional $\Sigma_P$ is a
generalised fan, that is, all the cones of $\Sigma_P$ contain a fixed
vector subspace $\sigma_0$ with associated lattice
$N_0=N\cap \sigma_0$. We denote by $F_P$ the toric variety for the fan
$\Sigma_P/\sigma_0 \subset (N/N_0)_\RR$.
\medskip

We recall two well-known facts about toric varieties which we use repeatedly:

\paragraph{Fact 1}

There is a 1-to-1 correspondence between:
\begin{enumerate}[(I)]
\item the set of polarised toric varieties $(F, D)$, that is, pairs of
  a proper toric variety $F$ and (torus invariant) ample divisor $D$,
  and
\item the set of full-dimensional rational polytopes $P\subset M_{\RR}$.
\end{enumerate}
Given a pair $(F,D)$, the corresponding polyhedron is
$P=P_D$. Conversely, given $P$, we let $F=F_P$ and $D=\sum b_iD_i$
where
$P=\bigcap_{i} \{ m \colon \langle \rho_i, m \rangle \geq -b_i \}$ is
the unique facet presentation of $P$.

\paragraph{Fact 2}

Suppose that $D$ is a nef divisor on a proper toric variety $F$. Let
$P=P_D$ be the polytope of $D$ and $(F_P, D_P)$ the corresponding
polarised pair. Then $(F_P, D_P)$ is an \emph{ample model} of
$(F, D)$, in other words:
\begin{enumerate}[(i)]
\item There is a proper toric morphism $f\colon F\to F_P$, and
\item $D=f^\star D_P$.
\end{enumerate}
\smallskip

\subsection{The toric variety associated to \texorpdfstring{$W$}{W}} Let
$M= \mathrm{Hom}(\TT^4, \CC^{\times})$ be the group of characters of
the torus $\TT^4$ in Equation \eqref{eq:2a} and let $N$ be its dual
lattice. Denote by $\{e_1, e_2, e_3, e_4\}$ the basis of $N$ dual to
the basis of $M$ consisting of the coordinate functions
$u_1, u_2, u_3,u_4$ on $\TT^4$.

In this Section we study the Newton polytope
$P \subset M_{\RR}$ of the polynomial in Equation \eqref{eq:3-fold-a},
defining the variety $W \subset \TT^4$, and the compactification of $W$ in the toric variety
$F=F_P$.

\medskip

The polytope $P$ is a 4-dimensional lattice polytope with six vertices
$u_0=\underline{0}$, $u_i$ for $i \in \{1,2,3,4\}$ and
$u_5=(2,2k+1,2k+1,4k+1)$, and facet presentation:
\begin{equation*}
  \label{eq:polytopeP} P=\bigcap_{i=1}^{8} \{ m \in M_{\RR} \colon \langle \rho_i, m \rangle \geq -a_i\} \subset M_{\RR}
\end{equation*}  where 
$(\rho_i, a_i) \in N \times \ZZ$ are:
\begin{equation}  \label{rho and a}
  \begin{split}
 \left\{ \begin{array}{ll}  
(\rho_i, a_i)=& (e_i,0) \qquad \text{for} \quad  i=1,2,3,4 \\
(\rho_5,a_5) =& \left((4k+1,-1,-1,-1), 1\right)\\
(\rho_6, a_6)=& \left((-1,3,-1,-1), 1\right) \\
(\rho_7, a_7)=& \left((-1,-1,3,-1), 1\right)\\
(\rho_8, a_8)= &\left((-(4k+1),-(4k+1),-(4k+1), 4k+3)),(4k+1)\right)
         \end{array} \right.
     \end{split}
   \end{equation}
   Hence the fan $\Sigma=\Sigma_P$ is a complete fan with six $4$-dimensional
   cones and eight rays generated by the vectors $\rho_i$, and $F$
   is a compact $4$-dimensional toric variety. Note that $P$ does not
   depend on $a$ and so neither do $\Sigma$ and $F$.

   The compactification of $W$ in $F$ corresponds to an element of the linear system
   $|O_{F}(D)|,$ where $D= \sum_{i=1}^{8} a_i D_i$ is the ample divisor of $F$ associated to $P$ (as in Fact $1$ above).
   \smallskip

\begin{figure}[!ht] \centering
\begin{tikzpicture}[scale=0.8][line cap=round,line join=round,>=triangle 45,x=1cm,y=1cm]
   \clip(-9.5,-3) rectangle (16,4);
  \node at (2.7,-0.3) {\small{$\rho_1$}};   
  \node at (4.3,-0.9) {\small{$\rho_2$}};
  \node at(6.26,-0.45) {\small{$\rho_3$}};
  \node at (5.76,-2.74) {\small{$\rho_6$}};
  \node at ((8.7,-0.86) {\small{$\rho_7$}};
  \node at  (0.8,-1.28) {\small{$\rho_5$}}; 
  \node at  (3.74, 1.4) {\small{$\rho_4$}}; 
  \node at  (4.12, 3.5) {\small{$\rho_8$}}; 
   \node at (-6.85,-1.26) {\small{$u_1$}};   
  \node at (-3.4,-2.75){\small{$u_2$}};
 \node at (-1.15,-1.3) {\small{$u_3$}};
 \node at (-4.38,3.45){\small{$u_5$}};
  \node at  (-4.6,0.74)  {\small{$u_4$}};
  \node at   (-4.55,-0.4)  {\small{$u_0$}}; 

\draw [line width=1pt,color=black] (3.06,-0.34)-- (4.6,-0.7);
\draw [line width=1pt,color=black]  (4.6,-0.7)-- (6.26,-0.14);
\draw [line width=1pt,dash pattern=on 1pt off 1pt,color=black] (6.26,-0.14)-- (3.06,-0.34);
\draw [line width=1pt,color=black]  (4.16,1.28)-- (4.6,-0.7);
\draw [line width=1pt,color=black] (4.16,1.28)-- (3.06,-0.34);
\draw [line width=1pt,color=black]  (4.16,1.28)-- (6.26,-0.14);
\draw[line width=1pt,color=black]  (5.76,-2.42)-- (4.6,-0.7);
\draw [line width=1pt,color=black] (6.26,-0.14)-- (8.16,-0.86);
\draw [line width=1pt,color=black]  (8.16,-0.86)-- (5.76,-2.42);
\draw [line width=1pt,color=black]  (3.06,-0.34)-- (1.22,-1.28);
\draw [line width=1pt,color=black]  (1.22,-1.28)-- (4.12,3.1);
\draw[line width=1pt,color=black] (4.12,3.1)-- (4.16,1.28);
\draw [line width=1pt,color=black]  (4.12,3.1)-- (5.76,-2.42);
\draw [line width=1pt,color=black]  (5.76,-2.42)-- (1.22,-1.28);
\draw[line width=1pt,color=black]  (8.16,-0.86)-- (4.12,3.1);
\draw [line width=1pt,dash pattern=on 1pt off 1pt,color=black] (1.22,-1.28)-- (8.16,-0.86);
\draw[line width=1pt,color=black]  (-4.22,-0.56)-- (-6.48,-1.26);
\draw[line width=1pt,color=black]  (-4.22,-0.56)-- (-3.4,-2.42);
\draw [line width=1pt,color=black]  (-4.22,-0.56)-- (-1.6,-1.22);
\draw [line width=1pt,color=black]  (-4.22,-0.56)-- (-4.23,0.7);
\draw[line width=1pt,color=black]  (-4.23,0.7)-- (-4.38,3.1);
\draw [line width=1pt,color=black]  (-4.38,3.1)-- (-6.48,-1.26);
\draw [line width=1pt,dash pattern=on 1pt off 1pt,color=black] (-6.48,-1.26)-- (-1.6,-1.22);
\draw[line width=1pt,color=black] (-1.6,-1.22)-- (-3.4,-2.42);
\draw [line width=1pt,color=black]  (-3.4,-2.42)-- (-6.48,-1.26);
\draw [line width=1pt,color=black]  (-6.48,-1.26)-- (-4.23,0.7);
\draw [line width=1pt,color=black]  (-4.23,0.7)-- (-1.6,-1.22);
\draw[line width=1pt,color=black]  (-1.6,-1.22)-- (-4.377,3.1);
\draw [line width=1pt,color=black]  (-4.35,3.1)-- (-3.37,-2.42);
\draw [line width=1pt,color=black] (-4.23,0.7)-- (-3.40,-2.42);
\begin{scriptsize}
\draw [fill=black] (3.06,-0.34) circle (2.5pt);
\draw [fill=black](4.6,-0.7) circle (2.5pt);
\draw [fill=black] (6.26,-0.14) circle (2.5pt);
\draw[fill=black] (4.16,1.28) circle (2.5pt);
\draw [fill=black] (5.8,-2.42) circle (2.5pt);
\draw [fill=black] (8.16,-0.86) circle (2.5pt);
\draw [fill=black](1.22,-1.28) circle (2.5pt);
\draw [fill=black](4.12,3.1) circle (2.5pt);
\draw[fill=black] (-4.22,-0.56) circle (2.5pt);
\draw [fill=black] (-6.48,-1.26) circle (2.5pt);
\draw [fill=black] (-3.4,-2.42) circle (2.5pt);
\draw [fill=black] (-1.6,-1.22) circle (2.5pt);
\draw [fill=black] (-4.23,0.7) circle (2.5pt);
\draw [fill=black] (-4.377,3.1) circle (2.5pt);
\end{scriptsize}
\end{tikzpicture}
\caption{On the left, a picture of the polytope $P \subset M_{\RR}$;
 on the right, a picture of the normal fan $\Sigma \subset N_\RR$ with rays $\rho_i$ in $N_\RR$.}
\end{figure}
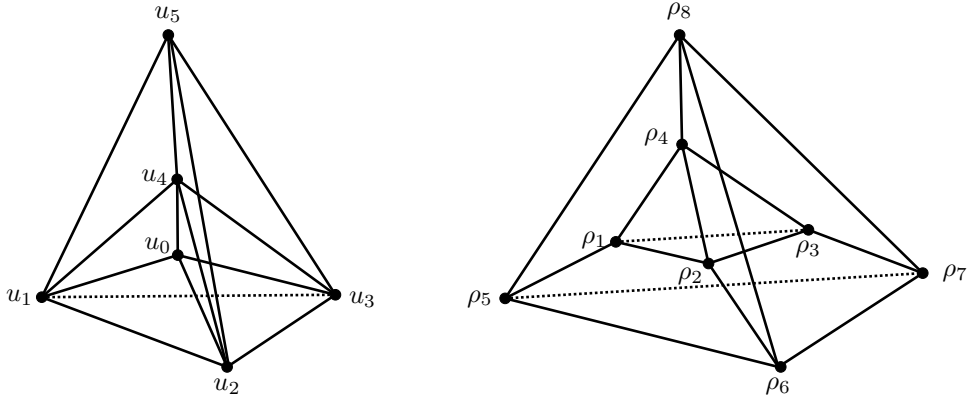

\begin{rem}
  The variety $F$ is not $\QQ$-Gorenstein. Let $F^\prime\to F$ be any
  small birational morphism such that $K_{F^\prime}$ is
  $\QQ$-Cartier. Then $F^\prime$ has noncanonical
  singularities. Indeed, consider for instance the affine open subsets
  $F^\prime_{\tau_1}$ and $F^\prime_{\tau_2}$ corresponding to the
  cones
 \begin{equation*}
   \tau_1 = \langle \rho_5, \rho_8\rangle_+
\quad \text{and} \quad
 \tau_2= 
\langle \rho_4, \rho_8\rangle_+
 \end{equation*}
 of the fan of $F^\prime$. Then 
\begin{equation*}
  F^\prime_{\tau_1} \cong \Gm^2 \times \frac1{4k+2}(1, 1)
\quad \text{and} \quad
  F^\prime_{\tau_2} \cong \Gm^2 \times \frac1{4k+1}(1, 2k)
\end{equation*}
The key point here is that the surface quotient singularities $\frac1{4k+2}(1, 1)$
and $\frac1{4k+1}(1, 2k)$ are not canonical:
the vectors $\frac1{4k+2}(1, 1)$ and $\frac1{4k+1}(1, 2k)$
correspond to valuations with discrepancies $-\frac{4k}{4k+2}$ and
$-\frac{2k}{4k+1}$. The weighted blow ups with weights $\frac1{4k+2}(1, 1)$
and $\frac1{4k+1}(1, 2k)$ are the minimal
canonical partial resolutions of these singularities.

Our aim is to construct a birational model for $F$ with canonical
singularities, as final product of a variant of the minimal model
program. 
\end{rem}

\subsection{A Mori fibre space structure} 

\paragraph{The plan}
In this Section our plan is to: 
\begin{enumerate}[(i)]
\item choose and construct a $\QQ$-Gorenstein partial resolution
  $f\colon \widetilde{F}\to F$ of $F$;
\item run the minimal model
  program with scaling for $(\widetilde{F},\widetilde{D})$, where $\widetilde{D}=f^\star (D)$;
\item study the singularities of the final product $(\overline{F},
  \overline{D})$ and the equation of $W$ in $\overline{F}$. 
\end{enumerate}

\subsubsection{The $\QQ$-Gorenstein partial resolution}


Set $\rho_9 = \frac1{4k+2}(\rho_5+\rho_8)=(0,-1,-1,1)$ and $a_9=1$.
(For $i \in \{1, \dots, 8\}$ let $a_i$ be as in~\eqref{rho and a}.) For
$\varepsilon \geq 0$ consider the polytope:
\begin{equation} \label{def-polytopes}
  P(\varepsilon)=\bigcap_{i=1}^{9} \{ m \in
  M_{\mathbb{R}}: \ \langle \rho_i, m_i\rangle \geq - a_i+
  \varepsilon\} \subset M_{\RR}
\end{equation} 
For $0<\varepsilon \ll 1$, $P(\varepsilon)$ is $4$-dimensional and the presentation in
Equation~\eqref{def-polytopes} is the facet presentation of
$P(\varepsilon)$, and the normal fan of $P(\varepsilon)$ is
independent of $\varepsilon$; denote it by  $\widetilde{\Sigma}$ and let
$\widetilde{F}=F_{P(\varepsilon)}$ be the corresponding toric variety. Let $\widetilde{D}=\sum_{i=1}^{9}a_i D_i$; by Fact~1 above the divisor
\begin{equation} \label{eq:def-divisors} \widetilde{D}(\varepsilon)=
  \sum_{i=1}^{9} (-\varepsilon +a_i) D_i =\varepsilon
  K_{\widetilde{F}} + \widetilde{D}
\end{equation}
is ample on $\widetilde{F}$.

\begin{lem}
  \label{lem:qGpartialresolution}
  \begin{enumerate}[(a)]
  \item $\widetilde{F}$ is $\QQ$-Gorenstein and the obvious map
    $f\colon \widetilde{F}\dasharrow  F$ is a morphism;  
  \item $\widetilde{D}=f^\star (D)$. 
  \end{enumerate}
\end{lem}
\begin{proof}[Sketch of proof]
Note that $\widetilde{D}$ is nef and $\widetilde{D}=\widetilde{D}(0)$.
For $\varepsilon=0$ we have that 
\[P(0)=\bigcap_{i=1}^{9} \{ m \in M_{\mathbb{R}}: \ \langle \rho_i,
m_i\rangle \geq - a_i\}= \bigcap_{i=1}^{8} \{ m \in M_{\mathbb{R}}: \
\langle \rho_i, m_i\rangle \geq - a_i\}= P
\]
where the second equality above follows from the choice $a_9=1$; thus
by Fact~2 above $(F,D)$ is the ample model of
$(\widetilde{F},\widetilde{D})$: in other words there is a morphism
$f\colon \widetilde{F}\to F$ and $\widetilde{D}= f^{\star}(D_P)$.
Since for $0<\varepsilon \ll 1$ the divisor
$\widetilde{D}(\varepsilon)=\varepsilon K_{\widetilde{F}} +
\widetilde{D}$
is ample on $\widetilde{F}$,
$K_{\widetilde{F}}\sim_f 1/\varepsilon \widetilde{D}(\varepsilon)$ is
$f$-ample. Finally, one can check explicitly that $\widetilde{F}$ is
$\QQ$-Gorenstein.
\end{proof}

\subsubsection{The minimal model program}

For $0< \varepsilon \ll 1$ the divisors $\widetilde{D}(\varepsilon)$
are ample on $\widetilde{F}$, since $K_{\tilde{F}}$ is $f$-ample and
$D \subset F$ is ample. 
We recover $\widetilde{F}$
as the toric variety whose spanning fan is the normal fan of the
polytopes $P(\epsilon)$ for small values of $0<\varepsilon$.

The minimal model program with scaling consists of an inductively
defined finite sequence of toric varieties $F_j$, divisors $D_j$, and a
strictly increasing sequence of rational numbers $\varepsilon_j$, and
birational maps:
\[
 F_0 \DashedArrow[densely dashed  ] \cdots \DashedArrow[densely dashed  ] F_j
\overset{t_j}{\DashedArrow[densely dashed  ]} F_{j+1}\DashedArrow[densely dashed  ] \cdots \DashedArrow[densely dashed  ] F_r=\overline{F}
\]
where:
\begin{enumerate}
\item We start with $F_0=\widetilde{F}$, $D_0=\widetilde{D}$, and $\varepsilon_0=
\max \{\varepsilon>0 \mid \varepsilon K_{\widetilde{F}}+\widetilde{D}
\; \text{is nef on $\widetilde{F}$}\}$;
\item For $j\geq 0$, $t_j\colon F_j \dasharrow F_{j+1}$ is the divisorial contraction
or flip of the face $R_j\subset \mathrm{NE} (F_j)$ with
$(\varepsilon_jK_j+D_j)_{|R_j} =0$ and $(K_j+D_j)_{|R_j}<0$;
\item For $j>0$ $D_j$ is the proper transform of $D_{j-1}$ on $F_j$,
  $K_j=K_{F_j}$, and:
\[
\varepsilon_j=\max \{\varepsilon>\varepsilon_{j-1} \mid \varepsilon
K_j+D_j \; \text{is nef on $F_j$}\}
\]
\item The program ends at $F_r=\overline{F}$ where \textbf{either}:
  \begin{description}
  \item[The pair $(\overline{F}, \overline{D})$ is a minimal model] that is
    $\overline{\varepsilon}=1$ and $K_{\overline{F}} + \overline{D}$ is nef; \textbf{or}
  \item[The pair $(\overline{F},\overline{D})$ is a Mori fibre space (Mfs)]
    that is $\overline{\varepsilon}<1$, the 
    contraction $\psi \colon \overline{F} \to S$ of $R_r$ has relative
    dimension $>0$ and $K_{\overline{F}}+\overline{D}$ is $\psi$-ample. 
    In this case $\overline{\varepsilon}$ is the quasi-effective threshold of the pair
    $(\widetilde{F}, \widetilde{D})$:
\[
\overline{\varepsilon}=\sup \{t \mid tK_{\widetilde{F}}+\widetilde{D} \in
\Effbar (\widetilde{F})\}
\]
  \end{description}
\end{enumerate}

By Fact~1, we recover $F_j$ as the toric variety whose spanning fan is the normal
fan of the polytopes $P(\varepsilon)$ for $\varepsilon \in
(\varepsilon_{j-1}, \varepsilon_j)$, since the divisors  $\varepsilon K_j +D_j$ are ample on $F_j$; the threshold values $\cdots
<\varepsilon_j<\varepsilon_{j+1}<\cdots$ are those where the polytope
$P(\varepsilon)$ changes shape. 

\subsubsection{The final product}
Returning to the polytopes of Equation~\eqref{def-polytopes}, $P(1)=
\emptyset$, thus the minimal model program just described will end with a Mfs. The sequence of threshold values is\footnote{We computed the threshold
  values as follows. Consider the
$5$-dimensional lattice $N^\prime=N \oplus \ZZ$ and the dual lattice
$M^\prime$, whose elements are pairs $(m, r)$ with $m\in M$ and $r\in
\ZZ$; also denote by $q \colon M^\prime\to \ZZ$ the projection to the
second factor. Let $\rho'_i=(\rho_i,-1) \in N'$, $i=1, \dots, 10$ and
let $P^\prime$ be the $5$-dimensional polyhedron:
\[
P'=\bigcap_{i=1}^{10}\{ m' \in {(M')}_{\RR} : \langle m',
\rho'_i\rangle \geq -a_i\} \subset {(M')}_{\RR}
\]
Then $\forall \varepsilon$ $P(\varepsilon)=q_{\RR}^{-1}(\varepsilon)\cap P^\prime$ and
the threshold values occur at the vertices of $P^\prime$.  
}:
\[
\frac1{2}, \frac{4k+1}{6k+2}, \frac{2}{3}
\]
Thus $\overline{\varepsilon}=2/3$, 
\[
P\left(\frac{2}{3}\right)=\left[\left(\frac{2}{3},\frac{2}{3},\frac{2}{3},1\right),
 \left(\frac{2}{3},\frac{2k+1}{3},\frac{2k+1}{3},\frac{4k+1}{3}\right) \right]
\]
and $S=\PP^1$. More precisely, the normal fan of  $P\left(2/3 \right)$
has two maximal cones:
\[
  C_1= \{ n \in N_\RR: \ n_2+n_3+2n_4\geq 0\} 
\quad \text{and} \quad 
C_2=\{ n \in N_\RR: \ n_2+n_3+2n_4\leq 0\}
\]
intersecting along the hyperplane $C_0=\{ n \in N_\RR: \ n_2+n_3+2n_4=0\} \subset N_{\RR}$.

For all $\varepsilon \in \left((2k+1)/(3k+2), 2/3\right)$,
$P(\varepsilon)$ is a $4$-dimensional polytope with $9$ vertices and
$7$ facets, whose normals are generated by the vectors $\rho_1, \rho_2, \rho_3, \rho_5, \rho_6, \rho_7,\rho_{10}$;
thus its normal fan $\overline{\Sigma}$ has $7$ rays and $9$ maximal cones, which are listed in Figure \ref{fig:overline-Sigma}.
One can check that the $4$-dimensional toric variety $\overline{F}$ is
$\QQ$-Gorenstein but not $\QQ$-factorial, and has canonical
singularities.

The toric morphism $\psi \colon \overline{F} \to \PP^1$ is induced by the quotient map $N \to N/N_0$, where $N_0=N\cap C_0$, sending
\[(n_1,n_2,n_3,n_4) \mapsto n_2+n_3+2n_4
\]

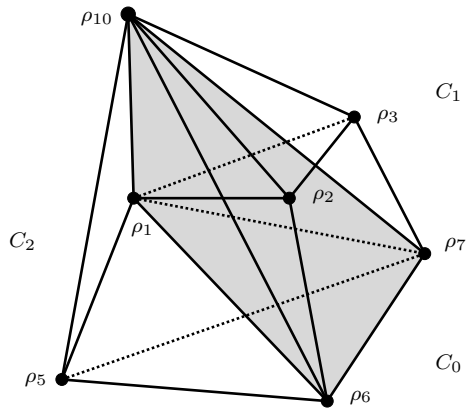
\begin{figure}[!ht] \centering
  \begin{tikzpicture}[scale=0.9][line cap=round,line
    join=round,>=triangle 45]
    \clip(-4,-2.2) rectangle (4,4.8);
    \begin{scriptsize}
      \draw [fill=black] (-6.5,0.5) node {$ \langle \rho_{1},  \rho_{5},  \rho_{6}, \rho_{10} \rangle_{+}$}; 
      \draw [fill=black] (-6.5,0) node {$\langle \rho_{1},   \rho_{5},  \rho_{7}, \rho_{10} \rangle_{+}$};
      \draw [fill=black] (-6.5,-0.5) node {$ \langle \rho_{5}, \rho_{6}, \rho_{7}, \rho_{10} \rangle_{+}$};
      \draw [fill=black] (-6.55,-1) node {$\langle \rho_{1},\rho_{5}, \rho_{6}, \rho_{7}\rangle_{+}$};
      \end{scriptsize}
    
\definecolor{uququq}{rgb}{0.25098039215686274,0.25098039215686274,0.25098039215686274}
\clip(-9,-2.08) rectangle (14,4.9);
\fill[line width=0pt,color=uququq,fill=uququq,fill opacity=0.2] (-1.66,1.44) -- (1.2,-1.56) -- (2.64,0.62) -- (-1.74,4.16) -- cycle; 
\draw[line width=1pt,color=black] (-1.66,1.44)-- (0.64,1.44); 
\draw[line width=1pt,color=black] (0.64,1.44)-- (1.6,2.64);
\draw[line width=1pt,color=black] (1.6,2.64)-- (-1.74,4.16);
\draw[line width=1pt,color=black] (-1.74,4.16)-- (-1.66,1.44);
\draw[line width=1pt,color=black] (-1.74,4.16)-- (-2.72,-1.24);
\draw[line width=1pt,color=black] (-2.72,-1.24)-- (1.2,-1.56);
\draw[line width=1pt,color=black] (1.2,-1.56)-- (2.64,0.62);
\draw[line width=1pt,color=black] (2.64,0.62)-- (-1.74,4.16);
\draw[line width=1pt,color=black] (0.64,1.44)-- (-1.74,4.16);
\draw[line width=1pt,color=black] (-1.74,4.16)-- (1.2,-1.56);
\draw[line width=1pt,color=black] (-1.66,1.44)-- (-2.72,-1.24);
\draw [line width=1pt,color=black, dash pattern=on 1pt off 1pt] (-2.72,-1.24)-- (2.64,0.62);
\draw[line width=1pt,color=black] (0.64,1.44)-- (1.2,-1.56);
\draw [line width=1pt, color=black, dash pattern=on 1pt off 1pt] (-1.66,1.44)-- (2.64,0.62);
\draw[line width=1pt,color=black](1.6,2.64)-- (2.64,0.62);
\draw [line width=1pt, color=black, dash pattern=on 1pt off 1pt](-1.66,1.44)-- (1.6,2.64);
\draw[line width=1pt,color=black] (-1.66,1.44)-- (1.2,-1.56);
\begin{scriptsize} 
\draw [fill=black] (-1.66,1.44) circle (2.5pt);
\draw [fill=black] (-1.53,0.98) node {$\rho_1$};
\draw [fill=black] (0.64,1.44) circle (2.5pt);
\draw [fill=black] (1.15,1.47) node {$\rho_2$};
\draw [fill=black] (1.6,2.64) circle (2.5pt);
\draw [fill=black] (2.1,2.64) node {$\rho_3$};

\draw [fill=black] (-1.74,4.16) circle (3pt);
\draw [fill=black] (-2.2,4.1) node {$\rho_{10}$};
\draw [fill=black] (-2.72,-1.24) circle (2.5pt);
\draw [fill=black] (-3.1,-1.24) node {$\rho_5$};
\draw [fill=black] (1.2,-1.56) circle (2.5pt);
\draw [fill=black] (1.7,-1.5) node {$\rho_6$};
\draw [fill=black] (2.64,0.62) circle (2.5pt);
\draw [fill=black] (3.1,0.75) node {$\rho_7$};

\draw [fill=black] (-3.3,0.8) node {$C_2$};

\draw [fill=black] (3,-1) node {$C_0$};

\draw [fill=black] ((3,3) node {$C_1$};
\end{scriptsize}

\begin{scriptsize} 
   \draw [fill=black] (6.68,3) node {$  \langle \rho_{1}, \rho_{2}, \rho_{3},  \rho_{6}, \rho_{7}\rangle_{+} $};
   \draw [fill=black] (6.75,2.5) node {$ \langle  \rho_{2}, \rho_{3},  \rho_{6}, \rho_{7}, \rho_{10}\rangle_{+}$};
   \draw [fill=black] (6.5,2) node {$\langle \rho_{1},  \rho_{3},  \rho_{7}, \rho_{10}  \rangle_{+}$};
   \draw [fill=black] (6.5,1.5) node {$  \langle \rho_{1}, \rho_{2},  \rho_{6},  \rho_{10}\rangle_{+}$};
   \draw [fill=black] (6.5,1) node {$ \langle \rho_{1}, \rho_{2}, \rho_{3}, \rho_{10}\rangle_{+}$};
  \draw [fill=black] (11,-1) node {};
\end{scriptsize}
\end{tikzpicture}
\caption{A picture of the fan $\overline{\Sigma} \subset N_\RR$ and of the cones $C_1 \ni \rho_2,\rho_3$ and $C_2 \ni \rho_5$. The vectors $\rho_1,\rho_6, \rho_7, \rho_{10}$ span the hyperplane $C_0 \subset N_\RR$, represented by the area coloured in gray. The cone $C_1$ is the union of the $5$ maximal cones of $\overline{\Sigma}$ containing $\rho_2$ or/and $\rho_3$, listed on the right of the picture. The cone $C_2$ is the union of the $4$ maximal cones of $\overline{\Sigma}$ containing $\rho_5$, listed on the left.} \label{fig:overline-Sigma}
\end{figure}

\medskip

The preimage via $\psi \colon \overline{F} \to \PP^1$ of the torus
$\CC^{\times} \subset \PP^1$ is the non-compact toric variety $G$
associated to the subfan $\Delta \subset \overline{\Sigma}$ given by the
cones of $\overline{\Sigma}$ contained in  $C_0$. The fan $\Delta$ has four $3$-dimensional
cones (\emph{see} Figure \ref{fig:overline-Sigma}):
\begin{equation*}
 \langle \rho_6, \rho_1, \rho_{10}
  \rangle_{+} \quad \langle \rho_7, \rho_1,
  \rho_{10}\rangle_{+} \quad \langle \rho_6, \rho_7, \rho_1
  \rangle_{+} \quad  \langle \rho_6, \rho_7, \rho_{10}
  \rangle_{+}
\end{equation*}
Since the quotient $N/N_0$ is torsion-free, we may choose a splitting
$N=N_0 \oplus N_1$ and regard $\Delta$ as a fan $\Delta_0$ in $N_0$;
we have $\Delta=\Delta_0\oplus \{0\}$, thus $
  G = G_{0} \times \CC^{\times}$,  where $G_{0}$ is the $3$-dimensional
toric variety with spanning fan $\Delta_0 \subset N_0$. To determine
$G_0$, note that the vectors $\rho_1, \rho_6, \rho_7, \rho_{10}$ spanning
the hyperplane $C_0$ satisfy the linear relation: $\rho_6+\rho_7 +2\rho_1+2\rho_{10}=0$.
However  $G_0$ is not the weighted projective space $\PP(1,1,2,2)$, as, for instance,  $1/2 \cdot (\rho_6+\rho_1 +\rho_{10}) \in N_0$.
Then write:
\begin{equation*}
N_{00}=\mathbb{Z} \rho_6 + \mathbb{Z} \rho_1 + \mathbb{Z} \rho_{10}
  \quad \text{and} \quad N_0=N_{00} + 1/2 \cdot (1,1,1) \mathbb{Z}
\end{equation*}
We have
$M_{00} \supset M_0=\{m \in M_{00}: \ m_1+m_2+m_3 =0 \ \mathrm{mod} \
2\}$
and $\CC[M_0]$ is the ring of invariants ${\CC[M_{00}]}^{\mu_2}$,
where the group where $\mu_2$ acts on $\CC[M_{00}]$ by
$\xi \cdot \chi^m= \xi^{m_1+m_2+m_3} \chi^m$.  Equivalently $G_0$ is the
quotient variety:
\begin{equation*}
  G_0 =  \PP(1,1,2,2) / \mu_2   
\end{equation*}
Denoting by $x_1,x_2,y_1,y_2$ the weighted homogeneous coordinates of
$\PP(1,1,2,2)$, the $\mu_2$ action has weights $1/2(1,1,1)$ on the two
smooth charts $\{x_1=1\}$ and $\{x_2=1\}$, and $1/4(1,-1,0)$ on the two
singular charts $\{y_1=1\}$ and $\{y_2=1\}$.
Thus $G$ is the toric variety  of Equation~\eqref{eq:G}.
\smallskip

To determine the equation of $W \subset \TT^4$ in $G=\PP(1,1,2,2)/\mu_2
\times \CC^{\times}$, we  choose the basis of $N_\RR$ given by the four vectors
$\rho_6, \rho_1, \rho_{10}, e_2$ and denote the dual
basis by $x$, $y$, $z$, $t$. In terms of $u_1$, $u_2$, $u_3$, $u_4$:
\begin{equation*}  
x=\left(0,0,-\frac{1}{2},-\frac{1}{2}\right) \quad 
y=\left(1,0,-\frac{1}{2},-\frac{1}{2}\right) \quad 
z=\left(0,0,-\frac{1}{2},\frac{1}{2}\right) \quad  
t=\left(0,1,1,2\right)
\end{equation*}
hence 
\[u_1=x^{-1} y \quad  
u_2=x^3z^{-1}t\quad   
u_3=x^{-1}z^{-1}\quad  
u_4=x^{-1}z\]
This is the substitution~\eqref{eq:change}, thus $W$ identifies with the \mbox{$3$-fold} of Equation \eqref{eq:W_new_equation} in $\TT_G$. The del Pezzo fibration $\phi \colon \widehat{W} \to \CC^{\times}$ of Section \ref{sec:proof} is the restriction of the toric morphism $\varphi \colon \overline{F} \to \PP^1$.

\subsection{A conic bundle structure}
\label{sec:conic-bundle-struct}

The del Pezzo fibration $\phi \colon \widehat{W} \to \Gm$ is not a
Mori fibre space in the Mori category: $\widehat{W}$ has strictly canonical
singularities along two sections of $\phi$ corresponding to the
orbifold points $p_t^\pm \in \widehat{W}_t$ (all $t$) of Lemma~\ref{lem:dPstructure}(3).

In this Section we construct a birational model $W^\prime$ with terminal singularities and a
Mori fibration $\psi \colon W^\prime \to \FF_1$ in the Mori
category. In our view, $\widehat{W}$ is a better model in which to
find the group $H^1(Y, \QQ)$.  

\paragraph{Construction}
Consider the toric variety $G^\prime$ with weight matrix\footnote{This
matrix defines an action of $\GG_m^3$ on $\AA^7$ and $G^\prime$ is a
GIT quotient.}:
\[
\begin{array}{ccccccc}
t_1 & t_2 & v_1 & v_2 & z_1 & z_2 & z \\
\hline
1 & 1 & 0 &-1& 0&-k-1&-k \\
0 & 0 & 1 & 1 & 0 & 0   &-2 \\
0 & 0 & 0 & 0 & 1 & 1   & 1
\end{array}
\]
and irrelevant ideal $(t_1, t_2)(v_1,v_2)(z,z_1,z_2)$. Note the morphism $G^\prime \to \FF_1$. Set \[\widetilde{\delta}_1(t_1,t_2)=4at_1^{2k+1} +t_2^{2k+1}\] and consider the hypersurface $W^\prime \subset G^\prime$ given by 
\begin{equation} \label{eq:conicbundle}
-z_1^2+t_2\widetilde{\delta}_1 z_2^2
+v_1v_2z^2\left(\widetilde{\delta}_1 v_1^2-at_1^{2k+1}t_2v_1v_2+t_1t_2\widetilde{\delta}_1 v_2^2\right)=0
\end{equation}

\begin{thm}
  \label{thm:conic_bundle}
  \begin{enumerate}[(A)]
 \item $W^\prime$ has terminal singularities and the morphism $\psi \colon W^\prime \to \FF_1$ is a
  conic bundle Mfs;
   \item There is a commutative diagram:
\[
\xymatrix{\widehat{W}\ar[d]_\phi  \ar@{-->}[r] & W^\prime \ar[d]^{\psi} \\
 \Gm & \FF_1 \ar@{-->}[l]}
\]
  \end{enumerate}
\end{thm}

\begin{proof} (A) It is possible, and not too hard, to verify explicitly that $W^\prime$
has isolated $\text{cA}$ singularities.

  (B) We start by exibiting an explicit birational map from a
  chart of $G$ to a chart of $G^\prime$. 

 Consider the chart
  $(t_2=v_2=z=1)\subset G^\prime$; this is isomorphic to $\CC^4$ with
  coordinate functions $t_1, v_1, z_1, z_2$. In this chart,
  $W^\prime$ is given by the equation:
\[
-z_1^2+\delta_1 z_2^2
+v_1\left(\delta_1 v_1^2-at_1^{2k+1} v_1+t_1 \delta_1 \right)=0
\]

Next consider the chart $(x_2=1)\subset G$; this is isomorphic to
$\GG_m \times \frac1{2}(1,1,1)$ with coordinate ring
\[
\CC[t]\otimes \CC[x_1,y_1,y_2]^{\mu_2}
\]
generated by the functions $t, x_1^2, x_1y_1$, \&c. In this
chart, $\widehat{W}$ is given by the equation:
\[
-\delta_1 y_1^2 + y_2^2 + x_1^4-\frac{at^{2k+1}}{\delta_1} x_1^2 +t =0
\]

We define a rational map from the chart in $G$ to the chart in $G^\prime$ by:
\[
v_1, z_1, z_2 \mapsto x_1^2, \delta_1 x_1 y_1, x_1y_2
\]
This map is in fact birational with inverse given by:
\[
x_1^2, x_1y_1, y_1^2, x_1y_2, y_1y_2, y_2^2 \mapsto v_1, \frac{z_1}{\delta_1},
\frac{z_1^2}{v_1 \delta_1^2}, z_2, \frac{z_1z_2}{v_1\delta_1}, \frac{z_2^2}{v_1}
\]
It is easy to see that the equation for $\widehat{W}$ is
transformed into the equation for $W^\prime$.
\end{proof}

\begin{rem}
  The variety $W^\prime$ has many singular points (more
    precisely, above $\widetilde{\delta}_1=v_j=0$, $t_1=v_1=0$,
    $t_2=v_j=0$, $j=1,2$); thus the conic bundle $\psi \colon W^\prime\to \FF_1$,
  although it is a Mori fibre space, is not standard (in
    the birational geometry literature, a \mbox{3-fold} Mfs $f\colon X
    \to S$ is a \emph{standard conic bundle} if $\dim S=2$ and $X$ is
    nonsingular, which implies that $S$ is also nonsingular, all fibres
  are conics, and the discriminant is normal crossing). Hence the
  conjectural rationality criteria
  \cite{MR1404188,MR1098841,MR1181206,MR899398} do not directly apply, though they  suggest that $W^\prime$ is nonrational (and possibly even
  birationally rigid). It would be interesting to study the geometry of
  $W^\prime$ further. 
   \end{rem}

  \begin{rem}
    The conic bundle $\psi \colon W^\prime \to \FF_1$ is an analog of
    a construction that in the singularity theory and mirror symmetry
    literature is called a ``double suspension.'' It is possible, of
    course, to find $H^1(Y,\QQ)$ in $H^3(W^\prime, \QQ)$, but it is
    easier to find it in $H^3(\widehat{W}, \QQ)$.
  \end{rem}


\bibliography{bib_gg}

\end{document}